\documentclass{article}
\usepackage{authblk}
\usepackage{hyperref}
\usepackage{amsmath} 
\usepackage{amsfonts}
\usepackage{mathptmx}
\DeclareMathOperator{\sign}{sign}
\usepackage{amsthm} 
\theoremstyle{definition}

\newtheorem{proposition}{Proposition}[section]
\newtheorem{theorem}{Theorem}[section]
\newtheorem{lemma}{Lemma}[section]

\numberwithin{equation}{section}
\usepackage[margin=1.5in]{geometry}
\providecommand{\msc}[1]{\small	\textbf{2020 Mathematics Subject Classification:} #1}
\providecommand{\keywords}[1]
{\small	\textbf{\textit{Keywords:}} #1}
\title{On the distribution of the telegraph meander and its properties}
\author[1]{A. Pedicone}
\author[2]{E. Orsingher}
\affil[ ]{Department of Statistical Sciences, Sapienza University of Rome, Italy}
\affil[1]{andrea.pedicone@uniroma1.it}
\affil[2]{enzo.orsingher@uniroma1.it}
\date{}
\begin{document}
\maketitle
\begin{abstract}
In this paper we study the telegraph meander, a random function obtained by conditioning the telegraph process to stay above the zero level. The finite dimensional distribution of the telegraph meander is derived by applying the reflection principle for the telegraph process and the Markovianity of the telegraph process with the velocity process. We show that the law of the telegraph meander at the end time is a solution to a hyperbolic equation, and we find the characteristic function and moments of any order. Finally, we prove that Brownian meander is the weak limit of the telegraph meander.
\end{abstract}
\\
\msc{Primary 60K99, Secondary 60F17}
\section{Introduction and summary}
Let $(\Omega, \mathcal{F},\mathsf{P})$ be a probability space and $N = (N(s))_{s\geq 0}$ be a homogeneous Poisson process on it. The telegraph process $T = (T(t))_{t\geq 0}$ is defined as the integral over $[0,t]$ of the process $V = (V(s))_{s\geq 0}$, $V(s) = V(0)(-1)^{N(s)}$, where $V(0)$ is a random variable independent of $N$ and assuming values $\{-c,c\}$, with equal probability. Formally, we have
\begin{equation}\label{telegraph process}
T(t) = \int_0^tV(0)(-1)^{N(s)}ds \;\; t \geq 0.
\end{equation}
The telegraph process has been introduced by Goldstein \cite{G 1951} and is the simplest example of the so-called \textit{random evolutions} (see \cite{EK 1986} Chapter 12). This random process formalizes the motion of a particle moving  at constant velocity and randomly changing direction, forward and backward, at exponential times. The finiteness of both the speed of motion and the intensity of changes of directions per unit time makes the telegraph process suitable for describing real motions that emerge in different fields, like geology \cite{TDc 2018}, finance \cite{DcP 2002, R 2007}, and physics \cite{HM 2020, MADlB2012}. For this reason over the years many researchers studied the telegraph process and its extensions see, e.g., \cite{O 1990}, \cite{O 1995}, \cite{R 1999}, \cite{R 2021}, \cite{BNO 2001}, \cite{DGO 2004}, \cite{DgI 2021} ,\cite{Dc 2001}, \cite{C 2023}, \cite{CO 2021}.

By means of different approaches \cite{O 1990, DGO 2004}, the distribution of $T(t)$ can be computed and reads
\begin{align}\label{law of telegraph process}&\mathsf{P}\{T(t) \in dx\} = \frac{e^{-\lambda t}}{2}\big(\mathsf{1}_{\{ct\}}(x)+\mathsf{1}_{\{-ct\}}(x)\big)dx+p_{\lambda,c}(x,t)dx, \end{align}
where 
\begin{equation}
p_{\lambda,c}(x,t) = \frac{e^{-\lambda t}}{2c}\Big(\lambda I_0\big(\frac{\lambda}{c}\sqrt{c^2t^2-x^2}\big)+\frac{\partial}{\partial t}I_0\big(\frac{\lambda}{c}\sqrt{c^2t^2-x^2}\big)\Big)\mathsf{1}_{(-ct,ct)}(x),
\end{equation}
and $I_\nu$ is the modified Bessel function of the first kind of order $\nu \in \mathbb{R}$, $I_\nu(x) = \sum_{k\geq 0}\big(\frac{x}{2}\big)^{2k+\nu}\frac{1}{k!\Gamma(\nu +k+1)}$, $x \in \mathbb{C}$. 

Next, we define the telegraph meander $T^+ = (T^+(s),s \in [0,t])$, for some $t>0$, as the stochastic process having finite dimensional distributions
\begin{align}\label{finite dimensional distribution}
&\mathsf{P}\{T^+(t_1)\in A_1,...,T^+(t_n) \in A_n\} =\mathsf{P}\{T(t_1)\in A_1,...,T(t_n) \in A_n|\min_{0\leq s\leq t}T(s)\geq 0\}\nonumber\\
&=\int_{A_1}...\int_{A_n}\sum_{\substack{v_1,...,v_n\\\in \{-c,c\}}}\prod_{k=1}^nf_{v_{k-1}}(x_k-x_{k-1},v_k,-x_{k-1},t_k-t_{k-1})\frac{m_{v_n}(-x_n,t-t_n)}{m_c(0,t)}dx_1 \cdots dx_n, 
\end{align}
for $0=t_0<t_1<...<t_n<t$, $v_0=c$, $x_0=0$, $A_1,..., A_n \in \mathcal{B}(\mathbb{R})$, where the functions $f$ and $m$ are defined in \eqref{notation 1} and \eqref{notation 2} respectively and are written explicitly in Lemma \ref{joint conditioned distribution telegraph with minimum} and in equations \eqref{minimum positive} and \eqref{minimum negative}. Hence, the telegraph meander corresponds to a telegraph process conditioned to stay above the zero level on the whole time interval $[0,t]$. From equation \eqref{telegraph process} it clearly follows that 
\begin{equation}\label{continuity}
\mathsf{P}\{|T^{+}(t)-T^{+}(s)|\leq c|t-s|\}=1.
\end{equation}
Therefore, we can define the telegraph meander as the random element $T^+:\Omega \to C^+[0,t]$, where the codomain is the space of non-negative continuous functions on $[0,t]$ endowed with the sigma field generated by the topology of the uniform metric. 

The study of the telegraph meander is possible thanks to the recent results of Cinque in \cite{C 2023} and \cite{C 2020}, where the joint distribution of the telegraph process and its maximum is presented. In particular the author proved a reflection principle for finite velocity random motions which, in the case of the telegraph process, becomes
\begin{equation}\label{Negative reflection 1}
\mathsf{P}\{T(t) \in dx,\max_{0\leq s \leq t}T(s) \geq y|V(0) = c\} = \mathsf{P}\{T(t) \in 2y-dx|V(0) = -c\},
\end{equation}
for $x \in [2y-ct,y)$ and $y \in [0,ct)$. The reflection principle was inspired by a previous result obtained by Cinque and Orsingher in \cite{CO 2021} where surprisingly the following result emerged
\begin{equation}\label{law of maximum}\mathsf{P}\{\max_{0\leq s\leq t}T(s) \in dx|V(0) = c\} = 2\mathsf{P}\{T(t) \in dx\} \;\; x \in [0,ct).\end{equation}

Here we show that the absolutely continuous component of the telegraph meander at the end time $t$ is 
\begin{align}
&q_{\lambda,c}(x,t) =\frac{\mathsf{1}_{[0,ct)}(x)}{I_0(\lambda t)+I_1(\lambda t))}\bigg\{\frac{\lambda xI_0\big(\frac{\lambda}{c}\sqrt{c^2t^2-x^2}\big)}{c(ct+x)}+\frac{I_1(\frac{\lambda}{c}\sqrt{c^2t^2-x^2})}{\sqrt{c^2t^2-x^2}}\Big(\frac{\lambda x}{c}+\frac{ct-x}{ct+x}\Big)\bigg\}.
\end{align}
We remark that this formula can also be obtained as a particular case of a result stated in \cite{R 2021}, where the joint law of the asymmetric telegraph process at the end time $t$, with its minimum at zero, is derived.

As a consequence of \eqref{Negative reflection 1}, we are able to express $q_{\lambda,c}$ as limit of the form
\begin{equation}\label{limit definition of telegraph meander}
q_{\lambda,c}(x,t) = \lim_{\epsilon \downarrow 0}\mathsf{P}\{T(t) \in dx |\min_{0\leq s \leq t} T(s) \geq -\epsilon,V(0) = -c\}.
\end{equation}
and we see that the following relationship emerges
\begin{equation}\label{Telegraph meander t 2}q_{\lambda,c}(x,t) = -\frac{\frac{\partial}{\partial x}p^{-}_{\lambda,c}(x,t)}{p^{-}_{\lambda,c}(0,t)} \;\; x \in [0,ct),\end{equation} where $p^{-}_{\lambda,c}$ is the absolutely continuous component of the law of telegraph process with initial negative velocity (see \cite{O 1995}),
\begin{align}\label{p minus}
&p^{-}_{\lambda,c}(x,t) =\frac{\lambda e^{-\lambda t}}{2c}\bigg(\sqrt{\frac{ct- x}{ct+x}}I_1\big(\frac{\lambda}{c}\sqrt{c^2t^2-x^2}\big)+I_0\big(\frac{\lambda}{c}\sqrt{c^2t^2-x^2}\big)\bigg)\mathsf{1}_{(-ct,ct)}(x).
\end{align}

Result \eqref{Telegraph meander t 2} permits us to prove that $q_{\lambda,c}$ satisfies the following hyperbolic equation
\begin{equation}\label{Telegrphic meander's equation}
\frac{\partial^2 u}{\partial t^2}-c^2\frac{\partial^2 u}{\partial x^2}+2\lambda\frac{\partial u}{\partial  t}-\frac{2a}{t}\frac{\partial u}{\partial t}=\Big(\frac{\lambda}{t}-\frac{2a}{ t^2}\Big)u,
\end{equation}
where 
\begin{equation}
a(t) = \frac{I_1(\lambda t)}{I_0(\lambda t)+I_1(\lambda t)}.
\end{equation}
An interesting application of finite velocity random motions is that they can be used to represent the solution of hyperbolic equations. Indeed, it is a well known result that $p_{\lambda,c}$ satisfies the telegraph equation
\begin{equation}\label{Telegrapher's equation}
\frac{\partial^2u}{\partial t^2}-c^2\frac{\partial^2 u}{\partial x^2}+2\lambda\frac{\partial u}{\partial  t} = 0.
\end{equation}
Then, \eqref{Telegrphic meander's equation} has a probabilistic interpretation in that its solution corresponds to the law of a conditioned finite velocity random motion. 

In \cite{O 1990} it is proved that, under the Kac's condition, namely when $\lambda \to +\infty$, $c \to +\infty$ such that $c^2/\lambda \to 1$, the density of the telegraph process converges to the density of Brownian motion. Then, by introducing the parameter $\alpha = \lambda = c^2$, the telegraph process $T_{\alpha}$ forms a sequence of random functions on $C[0,+\infty)$, the space of real valued continuous functions on $[0,+\infty)$. In \cite{EK 1986}, p. 451, it is claimed that $T_{\alpha} \Rightarrow B$ on $C[0,+\infty)$, where $B$ is Brownian motion and the symbol $\Rightarrow$ stands for weak convergence. This result motivates us to prove the weak convergence of the telegraph meander to Brownian meander, that is $T_{\alpha}^+ \Rightarrow B^+$ on $C^+[0,t]$, where we are denoting $B^+$ as the random function on $C^+[0,t]$ corresponding to Brownian meander (see e.g. \cite{DI 1977} for the definition of Brownian meander). By using the asymptotic expansion of $I_\nu$ for every $\nu \in \mathbb{N}_0$, we first show a local limit theorem: \begin{equation}
\lim_{\alpha \to +\infty}\mathsf{P}\{T^+_{\alpha}(t_1) \in dx_1,...,T^+_{\alpha}(t_n) \in dx_n\} = \mathsf{P}\{B^+(t_1) \in dx_1,...,B^{+}(t_n) \in dx_n\},
\end{equation}
for every $n \in \mathbb{N}$, $x_1,...,x_n \in \mathbb{R}$, $0<t_1<...<t_n<t$. The proof of weak convergence of the probability measures induced by $T^+_\alpha$ is more involved. The main idea is based on the observation that the telegraph process can be seen as a random walk with a random number of summands. Hence, we can use the result of Iglehart, stated in \cite{I 1974}, concerning a functional central limit theorem for random walk conditioned to stay positive. 

The telegraph meander formalizes a finite velocity random motion constrained to be positive, with the property that the Brownian meander is its weak limit. These features suggest to use the telegraph meander as alternative model to Brownian meander. For instance, in finance, the Brownian meander is applied as a model for financial options where the payoff depends on whether the asset’s price reaches a barrier and remains continuously below it (see \cite{CJ 1997}). Then, just as the geometric telegraph process replaces geometric Brownian motion (see \cite{DcP 2002}), we propose the telegraph meander as an alternative to Brownian meander in this field. Furthermore, since the telegraph process models the vertical motion of volcanic regions (\cite{TDc 2018}), the telegraph meander can be used for studying this phenomenon under the condition that the motion stays above a specified level.

This paper is organized in the following manner. In Section $2$ we derive the finite dimensional distribution of the telegraph meander and study the law of the telegraph meander at end time $t$. Then, in Section $3$ we study some properties of the law of $T^+(t)$, namely the characteristic function, the hyperbolic equation governing the law and the moments. Section $4$ concerns the weak convergence of the telegraph meander to Brownian meander. Finally, in Section $5$ we derive the law of the telegraph meander at end time $t$, conditioned on the number of Poisson events, and study some of its properties. 
\section{Distribution of the telegraph meander}
In order to derive the finite dimensional distribution of the telegraph meander we need to recall some of the results contained in \cite{CO 2021}, \cite{C 2023} and \cite{C 2020}. 

By setting as $\pi_v(dx,t,n) = \mathsf{P}\{T(t) \in dx|N(t) = n,V(0) = v\}$, for $n \in \mathbb{N}_0$ and $v \in \{-c,c\}$, according to \cite{CO 2021}, we have that
\begin{equation}
\pi_v(dx,t,n) =  
\begin{cases}\label{Telegraph conditioned even}
\mathsf{1}_{\{vt\}}(x)dx & n=0 \\
\frac{(2k+1)!}{k!^2}\frac{(c^2t^2-x^2)^{k}}{(2ct)^{2k+1}}\mathsf{1}_{(-ct,ct)}(x)dx &n = 2k+1,\;k \in \mathbb{N}_0\\
\frac{(2k)!}{k!(k-1)!}\frac{(c^2t^2-x^2)^{k}}{(2ct)^{2k}(ct-\sign(v)x)}\mathsf{1}_{(-ct,ct)}(x)dx &n = 2k,\;k \in \mathbb{N}.
\end{cases}
\end{equation}
Now, the symmetry of the telegraph process implies that, for any $x,\beta \in \mathbb{R}$ and $v \in \{-c,c\}$,
\begin{align}\label{symmetry property}
&\mathsf{P}\{T(t) \in dx,\min_{0\leq s \leq t}T(s) \geq -\beta |V(0) = -v\} =\mathsf{P}\{T(t) \in -dx,\max_{0\leq s \leq t}T(s) \leq \beta |V(0) = v\}. 
\end{align}
Therefore, formula (3.1) contained in \cite{C 2023} can be rewritten, according to our notation, for $n \in \mathbb{N}$, as
\begin{align}
&\mathsf{P}\{T(t) \in dx,\min_{0\leq s \leq t}T(s)\geq -\beta|N(t) = n,V(0) = -c\} \nonumber\\&=\begin{cases}\label{joint neg vel}
0 &\beta <0\lor -x\\
\pi_{c}(-dx,t,n)-\pi_{-c}(dx+2\beta,t,n)& 0\lor -x\leq \beta <\frac{ct-x}{2}\\
\pi_{c}(-dx,t,n) & \beta \geq \frac{ct-x}{2}
\end{cases} \;\; n \in \mathbb{N},
\end{align}
where we used the property that $\pi_c(-dx,t,n) = \pi_{-c}(dx,t,n)$. Furthermore, formulas (3.35) and (3.44) of \cite{C 2020} become respectively
\begin{align}
&\mathsf{P}\{T(t) \in dx,\min_{0\leq s \leq t}T(s)\geq -\beta|N(t) = 2k,V(0) = c\} \nonumber\\
&=\begin{cases}
0 & \beta<0\lor -x\\
\pi_{-c}(-dx,t,2k)-\pi_{-c}(dx+2\beta ,t,2k)& 0\lor -x\leq \beta <\frac{ct-x}{2}\\
\pi_{-c}(-dx,t,2k) & \beta \geq \frac{ct-x}{2}
\end{cases} \;\; k \in \mathbb{N}\label{joint pos vel 1}\\
&\mathsf{P}\{T(t) \in dx,\min_{0\leq s \leq t}T(s)\geq -\beta|N(t) = 2k+1,V(0) = c\} \nonumber\\
&=\begin{cases}
0 & \beta<0\lor -x\\
\pi_{-c}(-dx,t,2k+1)-\frac{k}{k+1}\frac{ct-(2\beta+x)}{ct+(2\beta+x)}\pi_{-c}(dx+2\beta,t,2k+1)& 0\lor -x\leq \beta <\frac{ct-x}{2}\\
\pi_{-c}(dx,t,2k+1) & \beta \geq \frac{ct-x}{2}
\end{cases} \;\; k \in \mathbb{N}_0.\label{joint pos vel 2}
\end{align}
Clearly we have
\begin{align}
&\mathsf{P}\{T(s) = v't,\min_{0\leq s \leq t}T(s)\geq -\beta|N(t) = 0,V(0) = v\} =\begin{cases}
\mathsf{1}_{[0,+\infty)}(\beta)\mathsf{1}_{\{c\}}(v') &v=c\\
\mathsf{1}_{[ct,+\infty)}(\beta)\mathsf{1}_{\{-c\}}(v') &v=-c.
\end{cases}
\end{align}
In the next lemma we derive the joint distribution of telegraph process, its minimum and velocity at time $t>0$ with fixed initial velocity.
\begin{lemma}\label{lemma joint distribution telegraph minimum with velocity}
Let $T$ be the telegraph process. Then, the following holds
\begin{align}\label{joint telegraph velocity minimum 1}
&\mathsf{P}\{T(t) \in dx,V(t) = v,\min_{0\leq s \leq t}T(s) \geq -\beta|V(0) = v\} \\
&=\begin{cases}
0 &\beta<0\lor -x\\
\frac{\lambda e^{-\lambda t}}{2c}\Big\{\sqrt{\frac{vt+x}{vt-x}}I_1(\frac{\lambda}{c}\sqrt{c^2t^2-x^2})-\sqrt{\frac{ct-(2\beta+x)}{ct+(2\beta+x)}}I_1(\frac{\lambda}{c}\sqrt{c^2t^2-(2\beta+x)^2})\Big\}dx &0\lor -x\leq \beta <\frac{ct-x}{2}\\
\frac{\lambda e^{-\lambda t}}{2c}\sqrt{\frac{vt+x}{vt-x}}I_1(\frac{\lambda}{c}\sqrt{c^2t^2-x^2}) & \beta\geq \frac{ct-x}{2}, \nonumber
\end{cases}
\end{align}
for $x \in (-ct,ct)$ and $v \in\{-c,c\}$, while the singularities are given by
\begin{align}
&\mathsf{P}\{T(t) = ct,V(t) = c,\min_{0\leq s \leq t}T(s) \geq -\beta|V(0) = c\} = e^{-\lambda t}\mathsf{1}_{[0,+\infty)}(\beta),\label{joint telegraph minimum 4} \\
&\mathsf{P}\{T(t) = -ct,V(t) = -c,\min_{0\leq s \leq t}T(s) \geq -\beta|V(0) = -c\} = e^{-\lambda t}\mathsf{1}_{[ct,+\infty)}(\beta)\label{joint telegraph minimum 5}.
\end{align}
Moreover, for $x \in (-ct,ct)$
\begin{align}
&\mathsf{P}\{T(t) \in dx,V(t) = c,\min_{0\leq s \leq t}T(s) \geq -\beta|V(0) = -c\} \label{joint telegraph velocity minimum 2} \\
&=\begin{cases}
0 &\beta<0\lor -x\\
\frac{\lambda e^{-\lambda t}}{2c}\big(I_0\big(\frac{\lambda}{c}\sqrt{c^2t^2-x^2}\big)-I_0(\frac{\lambda}{c}\sqrt{c^2t^2-(2\beta+x)^2})\big)dx & 0\lor -x\leq \beta <\frac{ct-x}{2}\\
\frac{\lambda e^{-\lambda t}}{2c}I_0\big(\frac{\lambda}{c}\sqrt{c^2t^2-x^2}\big)dx& \beta\geq \frac{ct-x}{2} \nonumber
\end{cases} \\
&\mathsf{P}\{T(t) \in dx,V(t) = -c,\min_{0\leq s \leq t}T(s) \geq -\beta|V(0) = c\} =  \label{joint telegraph velocity minimum 3}\\
&=\begin{cases}
0 &\beta<0\lor -x\\
\frac{\lambda e^{-\lambda t}}{2c}\big(I_0\big(\frac{\lambda}{c}\sqrt{c^2t^2-x^2}\big)-\frac{ct-(2\beta+x)}{ct+(2\beta+x)}I_2(\frac{\lambda}{c}\sqrt{c^2t^2-(2\beta+x)^2})\big)dx & 0\lor -x\leq \beta <\frac{ct-x}{2}\\
\frac{\lambda e^{-\lambda t}}{2c}I_0\big(\frac{\lambda}{c}\sqrt{c^2t^2-x^2}\big)dx &\beta \geq \frac{ct-x}{2}.\nonumber
\end{cases}
\end{align}
\end{lemma}
\begin{proof}
Results \eqref{joint telegraph minimum 4} and \eqref{joint telegraph minimum 5} are trivial. By using the definition of the $I_\nu$ and after simple steps, the following formulas can be obtained
\begin{align}
&\sum_{k \geq 1}\pi_{v}(dx,t,2k)\mathsf{P}\{N(t) = 2k\} = \frac{\lambda e^{-\lambda t}}{2c}\sqrt{-\frac{x+vt}{x-vt}}I_1\big(\frac{\lambda }{c}\sqrt{c^2t^2 -x^2}\big)dx,\label{joint telegraph and velocitiy 1}\\
&\sum_{k \geq 0}\pi_{v}(dx,t,2k+1)\mathsf{P}\{N(t) = 2k+1\} = \frac{\lambda e^{-\lambda t}}{2c}I_0\big(\frac{\lambda }{c}\sqrt{c^2t^2 -x^2}\big)dx,\label{joint telegraph and velocity 2}\\
&\sum_{k \geq 0}\frac{k}{k+1}\frac{ct-x}{ct+x}\pi_{-c}(dx,t,2k+1)\mathsf{P}\{N(t) = 2k+1\}= \frac{\lambda e^{-\lambda t}}{2c}\frac{ct-x}{ct+x}I_2(\frac{\lambda}{c}\sqrt{c^2t^2-x^2})dx,
\label{joint telegraph and velocity 3}
\end{align}
where $x \in (-ct,ct)$. Now, the velocity at the end time $t$ coincides with the initial velocity if there is an even number of changes of directions, so
\begin{align*}
&\mathsf{P}\{T(t) \in dx,V(t) = v,\min_{0\leq s \leq t}T(s) \geq -\beta|V(0) = v\} \\
&= \sum_{k\geq 0}\mathsf{P}\{T(t) \in dx,\min_{0\leq s \leq t}T(s) \geq -\beta|N(t) = 2k,V(0) = v\}\mathsf{P}\{N(t) = 2k\}.
\end{align*}
By using first \eqref{joint neg vel} for the even case and \eqref{joint pos vel 1}, by applying then \eqref{joint telegraph and velocitiy 1}, we arrive at \eqref{joint telegraph velocity minimum 1}. The proof of \eqref{joint telegraph velocity minimum 2} and \eqref{joint telegraph velocity minimum 3} works in the same way, using the fact that the velocity at the end time differs from the initial velocity if there is an odd number of changes of directions. This allows us to exploit \eqref{joint neg vel}, \eqref{joint pos vel 2}, and the formulas \eqref{joint telegraph and velocity 2}, \eqref{joint telegraph and velocity 3}. 
\end{proof}
The symmetry of $T$ implies also that the distribution of minimum can be obtained from the distribution of maximum. Hence, formulas contained in \cite{CO 2021} can be restated as
\begin{align}\label{minimum positive}
&\mathsf{P}\{\min_{0\leq s \leq t}T(s) \in dx |V(0) = c\} = e^{-\lambda t}(I_0(\lambda t)+I_1(\lambda t))\mathsf{1}_{\{0\}}(x)dx \nonumber\\
&+e^{-\lambda t}\bigg\{\frac{\lambda t}{ct-x}I_0\big(\frac{\lambda}{c} \sqrt{c^2t^2 -x^2}\big)+\sqrt{\frac{ct+x}{ct-x}}\Big(\frac{\lambda}{c}-\frac{1}{ct-x}\Big)I_1(\frac{\lambda }{c}\sqrt{c^2t^2 -x^2})\bigg\}\mathsf{1}_{[-ct,0)}(x)dx\\
&\mathsf{P}\{\min_{0\leq s \leq t}T(s) \in dx |V(0) = -c\} \nonumber\\&=e^{-\lambda t}\mathsf{1}_{\{-ct\}}(x)dx +\frac{e^{-\lambda t}}{c}\Big(\lambda I_0\big(\frac{\lambda}{c}\sqrt{c^2t^2-x^2}\big)+\frac{\partial}{\partial t}I_0\big(\frac{\lambda}{c}\sqrt{c^2t^2-x^2})\Big)\mathsf{1}_{(-ct,0]}(x)dx\label{minimum negative}
\end{align}
We have now all the necessary tools to obtain the finite dimensional distribution of telegraph meander. To simplify the notation we set
\begin{align}\label{notation 1}
&\mathsf{P}\{T(t) \in dx,V(t) =v,\min_{0\leq s\leq t}T(s) \geq \beta|V(0) = v'\} = f_{v'}(x,v,\beta,t)dx,\\
&\mathsf{P}\{\min_{0\leq s\leq t}T(s) \geq \beta|V(0) = v\} = m_{v}(\beta,t).\label{notation 2}
\end{align}
\begin{theorem}
Let $T$ be a telegraph process. Then, the following holds
\begin{align}\label{fdd}
&\mathsf{P}\{T(t_1) \in dx_1,...,T(t_n) \in dx_n|\min_{0\leq s\leq t}T(s) \geq 0\}\nonumber\\
&=\sum_{\substack{v_1,...,v_n\\\in \{-c,c\}}}\prod_{k=1}^nf_{v_{k-1}}(x_k-x_{k-1},v_k,-x_{k-1},t_k-t_{k-1})dx_k\frac{m_{v_n}(-x_n,t-t_n)}{m_c(0,t)}. 
\end{align}
where $0=t_0<t_1<...<t_n<t$, $x_0 = 0$, $v_0 = c$.
\end{theorem}
\begin{proof}
According to \cite{EK 1986}, p. 469,  the process $(T,V)$ is time homogeneous Markov process. This allows us to write, for every $n \in \mathbb{N}$, $0=t_0<t_1<...<t_n$ and $x_1,...,x_n \in \mathbb{R}$,
\begin{align*}
&\mathsf{P}\{T^+(t_1)\in dx_1,...,T^+(t_n) \in dx_n\}\\
&= \sum_{\substack{v_1,...,v_n\\\in \{-c,c\}}}\mathsf{P}\{T(t_1)\in dx_1,V(t_1) = v_1,...,T(t_n) \in dx_n,V(t_n) = v_n|\min_{0\leq s\leq t}T(s)\geq 0\}\\
&=\sum_{\substack{v_1,...,v_n\\\in \{-c,c\}}}\frac{\mathsf{P}\{\bigcap_{k=1}^n \{T(t_k)\in dx_k,V(t_k) = v_k,\min_{t_{k-1}\leq s\leq t_k}T(s)\geq 0\},\min_{t_n\leq s\leq t}T(s)\geq 0\}}{\mathsf{P}\{\min_{0\leq s\leq t}T(s)\geq 0\}}\\
&=\sum_{\substack{v_1,...,v_n\\\in \{-c,c\}}}\frac{\mathsf{P}\{\min_{t_n\leq s\leq t}T(s)\geq 0|T(t_n) = x_n,V(t_n) = v_n\}}{\mathsf{P}\{\min_{0\leq s\leq t}T(s)\geq 0|V(0) = c\}}\\&\mathsf{P}\{T(t_1)\in dx_1,V(t_1) = v_1,\min_{0\leq s\leq t_1}T(s)\geq 0|V(0) = c\}\\
&\prod_{k=2}^{n}\mathsf{P}\{T(t_k)\in dx_k,V(t_k) = v_k,\min_{t_{k-1}\leq s\leq t_k}T(s)\geq 0|T(t_{k-1}) = x_{k-1},V(t_{k-1}) = v_{k-1}\},
\end{align*}
where we used that $\{\min_{0\leq s \leq t} T(s) \geq 0\} \subseteq \{V(0) = c\}$. Finally we make use of the time homogeneity and the fact that $T(0) = 0$ $a.s.$ to get
\begin{align*}
&\mathsf{P}\{T^+(t_1)\in dx_1,...,T^+(t_n) \in dx_n\}\\
&=\sum_{\substack{v_1,...,v_n\\\in \{-c,c\}}}\frac{\mathsf{P}\{\min_{0\leq s\leq t-t_n}T(s)\geq -x_n|V(0) = v_n\}}{\mathsf{P}\{\min_{0\leq s\leq t}T(s)\geq 0|V(0) = c\}}\mathsf{P}\{T(t_1)\in dx_1,V(t_1) = v_1,\min_{0\leq s\leq t_1}T(s)\geq 0|V(0) = c\}\\
&\prod_{k=2}^{n}\mathsf{P}\{T(t_k-t_{k-1})\in dx_k-x_{k-1},V(t_k-t_{k-1}) = v_k,\min_{0\leq s\leq t_k-t_{k-1}}T(s)\geq -x_{k-1}|V(0) = v_{k-1}\}.
\end{align*}
\end{proof}
As a result of Theorem \ref{fdd}, Kolmogorov's existence theorem and the  continuity property \eqref{continuity} allow us to define $T^+$ as the stochastic process on $C^+[0,t]$ having finite dimensional distribution
\begin{align}\label{fdds}
&\mathsf{P}\{T^+(t_1)\in A_1,...,T^+(t_n) \in A_n\}= \mathsf{P}\{T(t_1)\in A_1,...,T(t_n) \in A_n|\min_{0\leq s\leq t}T(s)\geq 0\},
\end{align}
for $A_1,...,A_n \in \mathcal{B}(\mathbb{R})$.

We now examine the law of the telegraph meander at the end time $t$. By looking first at \eqref{joint telegraph minimum 4}, then at formulas \eqref{joint telegraph velocity minimum 1} when $v=c$ and  \eqref{joint telegraph velocity minimum 3}, setting $\beta = 0$ and adding up, we obtain
\begin{align}\label{joint distribution telegraph process and minimum}
&\mathsf{P}\{T(t) \in dx,\min_{0\leq s\leq t}T(s) \geq 0|V(0) = c\} = \nonumber\\
&=e^{-\lambda t}\mathsf{1}_{\{ct\}}(x)dx+\frac{\lambda e^{-\lambda t}}{2c}\bigg\{\frac{2xI_1(\frac{\lambda}{c}\sqrt{c^2t^2-x^2})}{\sqrt{c^2t^2-x^2}}+I_0\big(\frac{\lambda}{c}\sqrt{c^2t^2-x^2}\big)-\frac{ct-x}{ct+x}I_2(\frac{\lambda}{c}\sqrt{c^2t^2-x^2})\bigg\}\mathsf{1}_{[0,ct)}(x)dx\nonumber\\
&=e^{-\lambda t}\bigg\{\mathsf{1}_{\{ct\}}(x)+\bigg[\frac{\lambda xI_0\big(\frac{\lambda}{c}\sqrt{c^2t^2-x^2}\big)}{c(ct+x)}+\frac{I_1(\frac{\lambda}{c}\sqrt{c^2t^2-x^2})}{\sqrt{c^2t^2-x^2}}\Big(\frac{\lambda x}{c}+\frac{ct-x}{ct+x}\Big)\bigg]\mathsf{1}_{[0,ct)}(x)\bigg\}dx,
\end{align}
where in the last step we used the recurrence formula for Bessel function
\begin{equation}\label{properties Bessel}
I_{\nu+1}(x) = I_{\nu-1}(x)-\frac{2\nu}{x}I_\nu(x).
\end{equation}
Moreover, if the initial velocity is positive, the event $\{\min_{0\leq s\leq t}T(s)>0\}$ has null probability, then by using \eqref{minimum positive}
\begin{align}\label{Minimum non negative}
&\mathsf{P}\{\min_{0\leq s \leq t}T(s) = 0 |V(0) = c\} = \mathsf{P}\{\min_{0\leq s \leq t}T(s) \geq 0 |V(0) = c\} = e^{-\lambda t}(I_0(\lambda t)+I_1(\lambda t)).
\end{align}
Therefore, from \eqref{joint distribution telegraph process and minimum} and \eqref{Minimum non negative}, we can write
\begin{align}\label{Telegraph meander t 3}
&\mathsf{P}\{T^+(t) \in dx\} = \frac{\mathsf{P}\{T(t) \in dx,\min_{0\leq s\leq t}T(s) \geq 0|V(0) = c\}}{\mathsf{P}\{\min_{0\leq s\leq t}T(s) \geq 0|V(0) = c\}}\nonumber\\
&=\frac{\mathsf{1}_{\{ct\}}(x)dx}{I_0(\lambda t)+I_1(\lambda t)}+\frac{\mathsf{1}_{[0,ct)}(x)dx}{I_0(\lambda t)+I_1(\lambda t)}\bigg\{\frac{\lambda xI_0\big(\frac{\lambda}{c}\sqrt{c^2t^2-x^2}\big)}{c(ct+x)}+\frac{I_1(\frac{\lambda}{c}\sqrt{c^2t^2-x^2})}{\sqrt{c^2t^2-x^2}}\Big(\frac{\lambda x}{c}+\frac{ct-x}{ct+x}\Big)\bigg\}.
\end{align}

In the next theorem, by making use of the reflection principle for the telegraph process, we show that $q_{\lambda,c}$ (the absolutely continuous part of the law of the telegraph meander at time $t$), can be written as the limit of probabilities conditioned on events of the form ${\{\min_{0\leq s \leq t} T(s)\geq -\epsilon,V(0) = -c\}}$, for $\epsilon \downarrow 0$. By computing this limit, we can express $q_{\lambda,c}$ in terms of the derivative of the absolutely continuous part of the distribution of the telegraph process that starts with an initial negative velocity. Our construction mimics the strategy that Durrett \textit{et al.} followed for the derivation of the law of Brownian meander. We emphasize that this approach is possible thanks to the reflection principle, which by symmetry, can be stated as 
\begin{align}\label{Negative reflection}
&\mathsf{P}\{T(t) \in dx, \min_{0\leq s \leq t} T(s) < -\epsilon |V(0) = -c\} = \mathsf{P}\{T(t) \in 2\epsilon +dx|V(0) = -c\},
\end{align}
for $x \in[-\epsilon,ct-2\epsilon)$ and $\epsilon \in [0,ct)$. Result \eqref{Negative reflection}, beyond its usefulness, has an important meaning because it attributes to the telegraph process a typical property of Brownian motion. 
\begin{theorem}\label{Theorem telegraph meander t}
Let $T^+ = (T^+(s))_{s \in [0,t]}$ be a telegraph meander. Then, the absolutely continuous component of the law of $T^+(t)$ admits the following representation 
\begin{align}\label{Telegraph meander t 4}
&q_{\lambda,c}(x,t) = \lim_{\epsilon \downarrow 0}\mathsf{P}\{T(t) \in dx \lvert \min_{0\leq s \leq t} T(s) \geq -\epsilon,V(0) = -c\} =-\frac{\frac{\partial}{\partial x}p^{-}_{\lambda,c}(x,t)}{p_{\lambda,c}^{-}(0,t)} \;\; x \in [0,ct).
\end{align}
\end{theorem}
\begin{proof}
By symmetry, equation \eqref{law of maximum} assumes the form 
\begin{equation*}
\mathsf{P}\{\min_{0\leq s \leq t} T(s) \in -dx|V(0) = -c\} = 2\mathsf{P}\{T(t) \in dx\} \;\; x \in [0,ct).
\end{equation*} 
Moreover, for $x \in [-\epsilon,ct-2\epsilon)$ and $\epsilon \in [0,ct)$, result \eqref{Negative reflection} yields to
\begin{align*}
&\mathsf{P}\{T(t) \in dx | \min_{0\leq s \leq t} T(s) \geq -\epsilon,V(0) = -c\}\\&=\frac{\mathsf{P}\{T(t) \in dx, \min_{0\leq s \leq t} T(s) \geq -\epsilon |V(0) = -c\}}{\mathsf{P}\{\min_{0\leq s \leq t} T(s) \geq - \epsilon|V(0) = -c\}}\\
&=\frac{\mathsf{P}\{T(t) \in dx|V(0) = -c\}-\mathsf{P}\{T(t) \in dx, \min_{0\leq s \leq t} T(s) < -\epsilon |V(0) = -c\}}{\mathsf{P}\{\min_{0\leq s \leq t} T(s) \geq - \epsilon|V(0) = -c\}} \\
&=\frac{\mathsf{P}\{T(t) \in dx|V(0) = -c\}-\mathsf{P}\{T(t) \in 2\epsilon+dx|V(0) = -c\}}{\int_{0}^{\epsilon} 2 \mathsf{P}\{T(t) \in dy\}} = \frac{p^{-}_{\lambda,c}(x,t)-p^{-}_{\lambda,c}(2\epsilon+x,t)}{2\int_0^\epsilon p_{\lambda,c}(y,t)dy}.
\end{align*}
An application of the mean value theorem leads to the following asymptotic estimate
\begin{align}\label{asymptotic estimate formula}
&\mathsf{P}\{T(t) \in dx \lvert \min_{0\leq s \leq t} T(s) \geq -\epsilon,V(0) = -c\} \sim \frac{-2\epsilon \frac{\partial }{\partial x}p^{-}_{\lambda,c}(x,t)}{2\epsilon p_{\lambda,c}(0,t)}, \;\; \epsilon \downarrow 0,
\end{align}
Now $p_{\lambda,c}(0,t) =  p^{-}_{\lambda,c}(0,t) = e^{-\lambda t}\frac{\lambda}{2c}(I_0(\lambda t)+I_1(\lambda t))$, while the space derivative of $p^{-}_{\lambda,c}(x,t)$, stated in \eqref{p minus}, can be obtained by performing the following calculation
\begin{align}\label{proof limit telegraph meander 2}
&\frac{\partial}{\partial x}p^{-}_{\lambda,c}(x,t) =-\frac{e^{-\lambda t}}{2c}\bigg\{\frac{\lambda x}{c}\frac{I_1(\frac{\lambda}{c}\sqrt{c^2t^2-x^2})}{\sqrt{c^2t^2-x^2}} + \sqrt{\frac{ct+x}{ct-x}}\frac{ctI_1(\frac{\lambda}{c}\sqrt{c^2t^2-x^2})}{(ct+x)^2} \nonumber\\
&+\sqrt{\frac{ct-x}{ct+x}}\frac{\lambda x}{c\sqrt{c^2t^2-x^2}}\bigg(I_0\big(\frac{\lambda}{c}\sqrt{c^2t^2-x^2}\big) -\frac{c}{\lambda\sqrt{c^2t^2-x^2}}I_1(\frac{\lambda}{c}\sqrt{c^2t^2-x^2}) \bigg)\bigg\} \nonumber\\
&=-\frac{e^{-\lambda t}}{2c}\bigg\{I_0\big(\frac{\lambda}{c}\sqrt{c^2t^2-x^2}\big)\frac{\lambda x}{c(ct+x)}+\frac{I_1(\frac{\lambda}{c}\sqrt{c^2t^2-x^2})}{\sqrt{c^2t^2-x^2}}\Big(\frac{\lambda x}{c}+\frac{ct-x}{ct+x}\Big)\bigg\}.
\end{align}
We remark that in the previous steps we used \eqref{properties Bessel} and the property
\begin{equation}\label{properties derivative Bessel}
\frac{d}{dx}I_\nu(x) = I_{\nu-1}(x) - \frac{\nu}{x}I_\nu(x).
\end{equation}
In conclusion, we have proved that
\begin{align}\label{proof limit telegraph meander 3}
\lim_{\epsilon \downarrow 0}\mathsf{P}\{T(t) \in dx |\min_{0\leq s \leq t} T(s) \geq -\epsilon,V(0) = -c\} = -\frac{\frac{\partial}{\partial x}p^{-}_{\lambda,c}(x,t)}{p_{\lambda,c}^{-}(0,t)},
\end{align}
and hence from \eqref{proof limit telegraph meander 2} we see that the left hand side of \eqref{proof limit telegraph meander 3} coincides with \eqref{Telegraph meander t 3}.
\end{proof}
Theorem \ref{Theorem telegraph meander t} exhibits once more a connection between the telegraph process and Brownian motion. Indeed, we recall that the density of Brownian meander at the end point is
\begin{equation}
q(x,t) = \frac{x}{t}e^{-\frac{x^2}{2t}} \;\; x>0.
\end{equation}
Then, by denoting with $p$ the density of Brownian motion, it holds that
\begin{equation}\label{Brownian meander t 2}
q(x,t) = -\frac{\frac{\partial}{\partial x}p(x,t)}{p(0,t)} \;\; x>0.
\end{equation}
From \eqref{Telegraph meander t 4}, we can write the distribution of the telegraph meander, written explicitly in \eqref{Telegraph meander t 3}, in terms of the density $p^{-}_{\lambda,c}$ as follows
\begin{align}\label{law of telegraph meander}
\mathsf{P}\{T^+(t) \in dx\} = \frac{p^{-}_{\lambda,c}(ct,t)}{p^{-}_{\lambda,c}(0,t)}-\frac{\frac{\partial}{\partial x}p^{-}_{\lambda,c}(x,t)}{p_{\lambda,c}^{-}(0,t)}\mathsf{1}_{[0,ct)}(x),
\end{align}
where the first member in the right hand side represents the discrete component. From this equation it is easy to see that \eqref{Telegraph meander t 3} defines a proper probability distribution.
Moreover, we can also derive the distribution function of the telegraph meander at the end time
\begin{equation}\label{Telegraph meander cdf}
\mathsf{P}\{T^+(t)\leq x\} = 
\begin{cases}
0 &x< 0\\
1-\frac{\sqrt{\frac{ct-x}{ct+x}}I_1(\frac{\lambda}{c}\sqrt{c^2t^2-x^2})+I_0\big(\frac{\lambda}{c}\sqrt{c^2t^2-x^2}\big)}{I_0(\lambda t)+I_1(\lambda t)} &0\leq x<ct \\
1 &x\geq ct.
\end{cases}
\end{equation}
\section{Properties of Telegraph meander}
We can rewrite formula \eqref{p minus} as
\begin{equation}\label{p minus 2}
p_{\lambda,c}^{-}(x,t) = \frac{e^{-\lambda t}}{2c}\bigg(\lambda I_0\big(\frac{\lambda}{c}\sqrt{c^2t^2-x^2}\big)+\frac{\partial}{\partial t}I_0\big(\frac{\lambda}{c}\sqrt{c^2t^2-x^2}\big)+c\frac{\partial}{\partial x}I_0\big(\frac{\lambda}{c}\sqrt{c^2t^2-x^2}\big) \bigg),
\end{equation}
and we see that, in order to find the characteristic function of $T^+(t)$, we first need an explicit form of the following integral
\begin{equation}\label{G function}
G(\gamma,t) = \int_0^{ct}e^{i\gamma x}I_0\big(\frac{\lambda}{c}\sqrt{c^2t^2-x^2}\big)dx.
\end{equation}
\begin{lemma}\label{Lemma}
Let $|\gamma| <\lambda/c$, then the following result holds
\begin{align}\label{formula Lemma}
&\int_0^{ct}e^{i\gamma x}I_0\big(\frac{\lambda}{c}\sqrt{c^2t^2-x^2}\big)dx = c\frac{\sinh{(t\sqrt{\lambda^2-c^2\gamma^2})}}{\sqrt{\lambda^2-c^2\gamma^2}} + i c^2\gamma\int_0^t\frac{\sinh{(s\sqrt{\lambda^2-c^2\gamma^2})}}{\sqrt{\lambda^2-c^2\gamma^2}}I_0(\lambda (t-s))ds.
\end{align}
\begin{proof}
We first recall equation (25) stated in \cite{O 1990},
\begin{align}\label{Orsingher 90}
\frac{\partial^2}{\partial t^2}I_0\big(\frac{\lambda}{c}\sqrt{c^2t^2-x^2}\big) = c^2\frac{\partial^2}{\partial x^2}I_0\big(\frac{\lambda}{c}\sqrt{c^2t^2-x^2}\big)+ \lambda^2 I_0\big(\frac{\lambda}{c}\sqrt{c^2t^2-x^2}\big).
\end{align}
After multiplying both sides of equation \eqref{Orsingher 90} with $e^{i\gamma x}$ and integrating over the interval $(0,ct)$, we obtain
\begin{align*}
&\int_{0}^{ct}e^{i\gamma x}\frac{\partial^2}{\partial t^2}I_0\big(\frac{\lambda}{c}\sqrt{c^2t^2-x^2}\big)dx =c^2\int_{0}^{ct}e^{i\gamma x}\frac{\partial^2}{\partial x^2}I_0\big(\frac{\lambda}{c}\sqrt{c^2t^2-x^2}\big)dx+ \lambda^2 \int_{0}^{ct}e^{i\gamma x}I_0\big(\frac{\lambda}{c}\sqrt{c^2t^2-x^2}\big)dx.
\end{align*}
By interchanging the derivative with the integral on the left-hand side and integrating by parts the first term on the right-hand side, it follows that
\begin{align*}
&\frac{\partial^2}{\partial t^2}\int_{0}^{ct}e^{i\gamma x}I_0\big(\frac{\lambda}{c}\sqrt{c^2t^2-x^2}\big)dx =(\lambda^2-c^2\gamma^2)\int_{0}^{ct}e^{i\gamma x}I_0\big(\frac{\lambda}{c}\sqrt{c^2t^2-x^2}\big)dx + i c^2\gamma I_0(\lambda t).
\end{align*}
Therefore the function $G(\gamma,t)$ is a solution of the initial value problem
\begin{equation}\label{Cauchy problem}
\begin{cases}
\frac{\partial^2}{\partial t^2}G(\gamma,t) -(\lambda^2-c^2\gamma^2)G(\gamma,t) = ic^2\gamma I_0(\lambda t).  \\
G(\gamma,0) = 0 \\
\frac{\partial}{\partial t} G(\gamma,0)|_{t=0} = c.
\end{cases}
\end{equation}
We apply the $t$-Laplace transform of parameter $\beta>\lambda$. The second member of the equation in \eqref{Cauchy problem} can be obtained through the following formula
\begin{equation}\label{Laplace transform I0}
\int_0^{+\infty}e^{-\beta t}I_0(\lambda t)dt = \frac{1}{\sqrt{\beta^2-\lambda^2}}.
\end{equation}
For the first member, by setting $\mathcal{L}G(\gamma,t)=\int_0^{+\infty}e^{-\beta t}G(\gamma,t)dt$, we have that
\begin{align}\label{proof lemma 1}
&\int_0^{+\infty}e^{-\beta t}\frac{\partial^2}{\partial t^2}G(\gamma,t)dt = \bigg[e^{-\beta t}\frac{\partial}{\partial t}G(\gamma,t)\bigg]_{0}^{+\infty}+\beta\int_0^{+\infty}e^{-\beta t}\frac{\partial}{\partial t}G(\gamma,t)dt \nonumber\\
&=-c+\beta\bigg[e^{-\beta t}G(\gamma,t)\bigg]_{0}^{+\infty}+\beta^2\int_0^{+\infty}e^{-\gamma t}G(\gamma,t)dt = -c+\beta^2\mathcal{L}G(\gamma,t).
\end{align}
Where the fact that $e^{-\beta t}\frac{\partial}{\partial t}G(\gamma,t)$ and $e^{-\beta t}G(\gamma,t)$ tend to zero as $t\to +\infty$, if $\beta>\lambda$, can be deduced by the estimate $G(\gamma,t)\leq \frac{c}{\lambda}\sinh(\lambda t)$. 

Therefore by means of \eqref{Laplace transform I0} and \eqref{proof lemma 1}, the Laplace transform of $G(\gamma,t)$ takes the form
\begin{align}\label{Laplace-Fourier transform}
&\mathcal{L}G(\gamma,t)  =\frac{c}{\beta^2-\lambda^2-c^2\gamma^2}+\frac{ic^2\gamma}{(\beta^2-\lambda^2-c^2\gamma^2)\sqrt{\beta^2-\lambda^2}}.
\end{align}
And now we derive the inverse Laplace transform of $\mathcal{L}G(\gamma,t)$, that is the explicit form of the function $G(\gamma,t)$. From \eqref{Laplace-Fourier transform}
\begin{align*}
&\mathcal{L}G(\gamma,t) =\frac{c}{2\sqrt{\lambda^2-c^2\gamma^2}}\bigg(\frac{1}{\beta-\sqrt{\lambda^2-c^2\gamma^2}} -\frac{1}{\beta+\sqrt{\lambda^2-c^2\gamma^2}}\bigg)\\
&+\frac{ic^2\gamma}{2\sqrt{\lambda^2-c^2\gamma^2}}\frac{1}{\sqrt{\beta^2-\lambda^2}}\bigg(\frac{1}{\beta-\sqrt{\lambda^2-c^2\gamma^2}} -\frac{1}{\beta+\sqrt{\lambda^2-c^2\gamma^2}}\bigg).
\end{align*}
Then, by using again \eqref{Laplace transform I0}
\begin{align*}
&\int_0^{+\infty}e^{-\beta t}G(\gamma,t)dt \\&=\int_{0}^{+\infty}e^{-\beta t}\frac{c\sinh(\sqrt{\lambda^2-c^2\gamma^2})}{\sqrt{\lambda^2-c^2\gamma^2}}dt+\int_0^{+\infty}e^{-\beta t}I_0(\lambda t)dt \int_0^{+\infty}e^{-\beta t}\frac{ic^2\gamma\sinh(\sqrt{\lambda^2-c^2\gamma^2})}{\sqrt{\lambda^2-c^2\gamma^2}}dt  \\
&=\int_{0}^{+\infty}e^{-\beta t}\bigg(\frac{c\sinh(\sqrt{\lambda^2-c^2\gamma^2})}{\sqrt{\lambda^2-c^2\gamma^2}}+\int_{0}^t\frac{ic^2\gamma\sinh((t-s)\sqrt{\lambda^2-c^2\gamma^2})}{\sqrt{\lambda^2-c^2\gamma^2}} I_0(\lambda s)ds\bigg)dt.
\end{align*}

And so at the end formula \eqref{formula Lemma} appears
\begin{align*}
G(\gamma,t) = c\frac{\sinh{(t\sqrt{\lambda^2-c^2\gamma^2})}}{\sqrt{\lambda^2-c^2\gamma^2}} + ic^2\gamma\frac{\sinh(t\sqrt{\lambda^2-c^2\gamma^2})}{\sqrt{\lambda^2-c^2\gamma^2}} \ast I_0(\lambda t),
\end{align*}
where $*$ denotes the convolution operator.
\end{proof}
\end{lemma}
We also state the following proposition, which highlights the structure of the formula expressed in Lemma \ref{Lemma}.
\begin{proposition}
Let $|\gamma| <\lambda/c$, then the following result holds
\begin{equation}\label{Lemma 3}
\int_{0}^{ct}\cos(\gamma x)I_0\big(\frac{\lambda}{c}\sqrt{c^2t^2-x^2}\big)dx = \frac{c\sinh{(t\sqrt{\lambda^2-c^2\gamma^2})}}{\sqrt{\lambda^2-c^2\gamma^2}}.
\end{equation}
\end{proposition}
\begin{proof}
By exploiting the Mac-Laurin expansion of the cosine function we have 
\begin{align*}
&\int_{0}^{ct}\cos(\gamma x)I_0\big(\frac{\lambda}{c}\sqrt{c^2t^2-x^2}\big)dx = \sum_{k \geq 0}(-1)^k\frac{\gamma^{2k}}{(2k)!}\int_0^{ct}x^{2k}I_0\big(\frac{\lambda}{c}\sqrt{c^2t^2-x^2}\big)dx\\
&=\sum_{k \geq 0}(-1)^k\frac{\gamma^{2k}}{(2k)!}\frac{(ct)^{2k+1}}{2}\Gamma\Big(k+\frac{1}{2}\Big)\sum_{r\geq 0}\Big(\frac{\lambda t}{2}\Big)^{2r}\frac{1}{r!\Gamma(r+k+\frac{3}{2})}.
\end{align*}
Now, we use the duplication formula of the Gamma function, from which we can write that
\begin{align*}
&\int_{0}^{ct}\cos(\gamma x)I_0\big(\frac{\lambda}{c}\sqrt{c^2t^2-x^2}\big)dx=\sum_{k \geq 0}(-1)^k\frac{\gamma^{2k}}{k!}(ct)^{2k+1}\sum_{r\geq 0}\frac{(\lambda t)^{2r}}{r!}\frac{(r+k)!}{(2(r+k)+1)!}\\
&=\sum_{k \geq 0}ct(-1)^k\Big(\frac{\gamma c}{\lambda}\Big)^{2k}\sum_{m\geq k}\binom{m}{k}\frac{(\lambda t)^{2m}}{(2m+1)!}=\sum_{m\geq 0}\frac{ct(\lambda t)^{2m}}{(2m+1)!}\sum_{k=0}^m\binom{m}{k}\Big(-\frac{c^2\gamma^2}{\lambda^2}\Big)^{k}\\
&=\sum_{m\geq 0}\frac{ct(\lambda t)^{2m}}{(2m+1)!}\Big(1-\frac{c^2\gamma^2}{\lambda^2}\Big)^{m},
\end{align*}
and so, by recalling the Mac-Laurin expansion of the hyperbolic sine, the thesis follows. 
\end{proof}
In light of the previous proposition, we apply the Euler's formula to the first member of \eqref{formula Lemma}, from which we can state that
\begin{align}\label{odd part G function}
&\int_{0}^{ct}\sin(\gamma x)I_0\big(\frac{\lambda}{c}\sqrt{c^2t^2-x^2}\big)dx =  c^2\gamma\int_0^t\frac{\sinh{(s\sqrt{\lambda^2-c^2\gamma^2})}}{\sqrt{\lambda^2-c^2\gamma^2}}I_0(\lambda (t-s))ds \;\; |\gamma| <\frac{\lambda}{c}.
\end{align}
\begin{theorem}
Let $T^+ = (T^+(s))_{s \in [0,t]}$ be a telegraph meander. Then
\begin{align}
&\mathsf{E}[e^{i\gamma T^+(t)}] =1 + c\gamma(c\gamma+i\lambda )\frac{\sinh(t\sqrt{\lambda^2 -c^2 \gamma^2})+ic\gamma\int_0^t\sinh{(s\sqrt{\lambda^2 -c^2 \gamma^2})}I_0(\lambda(t-s))ds}{\lambda\sqrt{\lambda^2 -c^2 \gamma^2}(I_0(\lambda t)+I_1(\lambda t))} \nonumber\\
&+ic\gamma\frac{\cosh(t\sqrt{\lambda^2 -c^2 \gamma^2})-I_0(\lambda t)+ic\gamma\int_0^t\cosh(s\sqrt{\lambda^2 -c^2 \gamma^2})I_0(\lambda(t-s))ds}{\lambda(I_0(\lambda t)+I_1(\lambda t))} \;\;|\gamma| <\frac{\lambda}{c}.
\end{align}
\begin{proof}
By having in mind \eqref{law of telegraph meander} and formula \eqref{p minus 2}, partial integration leads to
\begin{align*}
&\mathsf{E}[e^{i\gamma T^+(t)}] =-\int_0^{ct}\frac{e^{i\gamma x}\frac{\partial}{\partial x}p^{-}_{\lambda,c}(x,t)}{p^{-}_{\lambda,c}(0,t)}dx + \frac{e^{i\gamma ct}p^{-}_{\lambda,c}(ct,t)}{p^{-}_{\lambda,c}(0,t)}= 1 + i\gamma\int_0^{ct} \frac{e^{i\gamma x}p^{-}_{\lambda,c}(x,t)}{p^{-}_{\lambda,c}(0,t)}dx\\
&=1 + \frac{i\gamma}{I_0(\lambda t)+I_1(\lambda t)}\bigg(\int_0^{ct} \frac{e^{i\gamma x}}{\lambda}\frac{\partial}{\partial t}I_0\big(\frac{\lambda}{c}\sqrt{c^2t^2-x^2}\big)dx+\int_0^{ct} \frac{e^{i\gamma x}c}{\lambda}\frac{\partial}{\partial x}I_0\big(\frac{\lambda}{c}\sqrt{c^2t^2-x^2}\big)dx\\&+\int_0^{ct} e^{i\gamma x}I_0\big(\frac{\lambda}{c}\sqrt{c^2t^2-x^2}\big)dx\bigg).
\end{align*}
By interchanging the derivative with the integral sign on the first term and integrating by parts the second we get
\begin{align}\label{cf telegraph meander}
&\mathsf{E}[e^{i\gamma T^+(t)}]=1 + \gamma\frac{i\frac{\partial}{\partial t}G(\gamma,t) -icI_0(\lambda t) +(c\gamma+i\lambda)G(\gamma,t)}{\lambda(I_0(\lambda t)+I_1(\lambda t))},
\end{align}
where $G(\gamma,t)$ is defined in \eqref{G function}. Now, by putting into \eqref{cf telegraph meander} the explicit expression of $G(\gamma,t)$, stated in Lemma \ref{Lemma}, together with its time derivative,
\begin{align*}
\frac{\partial}{\partial t}G(\gamma,t) = c \cosh(t\sqrt{\lambda^2-c^2 \gamma^2}) +ic^2\gamma\int_0^t \cosh((t-s)\sqrt{\lambda^2-c^2 \gamma^2})I_0(\lambda s) ds,
\end{align*}
after simple calculations, we get the desired result. 
\end{proof}
\end{theorem}
\begin{proposition}\label{proposition moment telegraph process}
Let $T^+ = (T^+(s))_{s \in [0,t]}$ be a telegraph meander. Then, for $p>0$
\begin{align}\label{Telegraph meander moments 2}
&\mathsf{E}[T^+(t)^{p}] =\Big(\frac{2c^2t}{\lambda}\Big)^{\frac{p}{2}}\frac{\Gamma(\frac{p}{2}+1)(I_{\frac{p}{2}}(\lambda t)+I_{\frac{p}{2}-1}(\lambda t)) -\frac{p}{2}\Gamma(\frac{p+1}{2})\sqrt{\frac{2}{\lambda t}}I_{\frac{p-1}{2}}(\lambda t)}{I_0(\lambda t)+I_1(\lambda t)}
\end{align}
\end{proposition}
\begin{proof}
From formulas \eqref{law of telegraph meander} and \eqref{p minus 2}, we integrate by parts to obtain
\begin{align*}
&\mathsf{E}[T^+(t)^{p}] = -\frac{1}{p^{-}_{\lambda,c}(0,t)}\int_{0}^{ct}x^{p}\frac{\partial}{\partial x}p^{-}_{\lambda,c}(x,t)dx +(ct)^{p} \frac{p^{-}_{\lambda,c}(ct,t)}{p^{-}_{\lambda,c}(0,t)}  \\
&=\frac{1}{I_{0}(\lambda t)+I_{1}(\lambda t)}\bigg(\int_{0}^{ct}px^{p-1}I_0\big(\frac{\lambda}{c}\sqrt{c^2t^2-x^2}\big)dx +\frac{c}{\lambda}\int_{0}^{ct}px^{p-1}\frac{\partial}{\partial x}I_0\big(\frac{\lambda}{c}\sqrt{c^2t^2-x^2}\big)dx \\&+\frac{1}{\lambda}\int_{0}^{ct}px^{p-1}\frac{\partial}{\partial t}I_0\big(\frac{\lambda}{c}\sqrt{c^2t^2-x^2}\big)dx \bigg) \\
&=\frac{1}{I_{0}(\lambda t)+I_{1}(\lambda t)}\bigg(\int_{0}^{ct}px^{p-1}I_0\big(\frac{\lambda}{c}\sqrt{c^2t^2-x^2}\big)dx +\frac{p(ct)^p}{\lambda t}\frac{cp(p-1)}{\lambda}\int_{0}^{ct}x^{p-2}I_0(\frac{\lambda}{c}\sqrt{x^2-c^2t^2})dx\\&+\frac{1}{\lambda}\frac{d}{dt}\int_{0}^{ct}px^{p-1}I_0(\frac{\lambda}{c}\sqrt{x^2-c^2t^2})dx - \frac{p(ct)^p}{\lambda t} \bigg).
\end{align*}
Now, we use the following formula 
\begin{align}\label{Lemma 1}
&\int_{0}^{ct}x^{p}I_0\big(\frac{\lambda}{c}\sqrt{c^2t^2-x^2}\big)dx = \frac{\Gamma(\frac{p+1}{2})}{2}\Big(\frac{2c^2t}{\lambda}\Big)^{\frac{p+1}{2}}I_{\frac{p+1}{2}}(\lambda t), \;\; p >0,
\end{align}
from which it follows that
\begin{align*}
&\mathsf{E}[T^+(t)^{p}]=\frac{1}{I_{0}(\lambda t)+I_{1}(\lambda t)}\bigg(\Gamma\Big(\frac{p}{2}+1\Big)\Big(\frac{2c^2t}{\lambda}\Big)^{\frac{p}{2}}I_{\frac{p}{2}}(\lambda t) -p\frac{c}{\lambda}\Gamma\Big(\frac{p+1}{2}\Big)\Big(\frac{2c^2t}{\lambda}\Big)^{\frac{p-1}{2}}I_{\frac{p-1}{2}}(\lambda t)  \\
&+\frac{p^2}{4\lambda t}\Gamma\Big(\frac{p}{2}\Big)\Big(\frac{2c^2t}{\lambda}\Big)^{\frac{p}{2}} I_{\frac{p}{2}}(\lambda t)+\Gamma\Big(\frac{p}{2}+1\Big)\Big(\frac{2c^2t}{\lambda}\Big)^{\frac{p}{2}}I_{\frac{p}{2}-1}(\lambda t)-\frac{p^2}{4\lambda t}\Gamma\Big(\frac{p}{2}\Big)\Big(\frac{2c^2t}{\lambda}\Big)^{\frac{p}{2}}I_{\frac{p}{2}}(\lambda t)\bigg),
\end{align*}
and then, after simple manipulations the thesis can be achieved.
\end{proof}
We remember that $I_{\frac{1}{2}}(x) = \sqrt{\frac{2}{\pi x}}\sinh(x)$, $I_{-\frac{1}{2}}(x) = \sqrt{\frac{2}{\pi x}}\cosh(x)$, then thanks to \eqref{Telegraph meander moments 2}, we can write the mean and the variance 
\begin{align}
&\mathsf{E}[T^+(t)] = \frac{c(e^{\lambda t}-I_0(\lambda t))}{\lambda(I_0(\lambda t)+I_1(\lambda t))} \label{Telegraph meander first moment}\\
&\mathsf{V}[T^+(t)] =  \frac{2tc^2}{\lambda}- c^2\frac{I^2_0(\lambda t)+e^{2\lambda t}+2(I_1(\lambda t)\sinh(\lambda t)-I_0(\lambda t)\cosh(\lambda t))}{\lambda^2(I_0(\lambda t)+I_1(\lambda t))^2} \label{Telegraph meander second moment}
\end{align}
We now compare Brownian meander with telegraph meander. We first show that the density function of Brownian meander satisfies a parabolic partial differential equation in which at the first member there is the heat operator $\frac{\partial^2}{\partial x^2}-\frac{\partial}{\partial t}$.
\begin{proposition}
Let $B^+ = (B^+(s))_{s \in [0,t]}$ be Brownian meander. Then the density of $B^+(t)$ satisfies
\begin{equation}\label{pde Brownian meander 2}
\frac{1}{2}\frac{\partial^2u}{\partial x^2} - \frac{\partial u}{\partial t}=- \frac{u}{2t}
\end{equation}
\end{proposition}
\begin{proof}
For simplicity we rewrite equation \eqref{Brownian meander t 2} as
\begin{equation}\label{eq Brownian meander t 2}
q(x,t)p(0,t) = -\frac{\partial}{\partial x}p(x,t).
\end{equation}
By multiplying by $1/2$ and applying the second order space derivative on both members of \eqref{eq Brownian meander t 2} we obtain that
\begin{align}\label{cp pde Brownian meander}
\frac{1}{2}\frac{\partial^2}{\partial x^2}q(x,t)p(0,t) = -\frac{1}{2}\frac{\partial^3}{\partial x^3}p(x,t).
\end{align}
From the fact that the density function $p$ satisfies heat equation
\begin{equation}\label{Heat equation}
\frac{1}{2}\frac{\partial^2 p}{\partial x^2} = \frac{\partial p}{\partial t}
\end{equation}
and from \eqref{eq Brownian meander t 2}, we can see that the right hand side of \eqref{cp pde Brownian meander} can be written as 
\begin{align*}
&-\frac{1}{2}\frac{\partial^3}{\partial x^3}p(x,t) = \frac{\partial}{\partial t}\Big\{-\frac{\partial}{\partial x}p(x,t)\Big\} =\frac{\partial}{\partial t}q(x,t)p(0,t)+q(x,t)\frac{\partial}{\partial t}p(0,t) =p(0,t)\Big\{\frac{\partial}{\partial t}q(x,t)-q(x,t)\frac{1}{2t}\Big\}
\end{align*}
where we used the fact $\frac{\partial}{\partial t}p(0,t) = -\frac{1}{2 t}p(0,t)$. Therefore, equation \eqref{cp pde Brownian meander} becomes 
\begin{equation*}
\frac{1}{2}\frac{\partial^2}{\partial x^2}p(0,t)q(x,t) =  p(0,t)\frac{\partial }{\partial t}q(x,t)-\frac{1}{2t}p(0,t)q(x,t),
\end{equation*}
that is \eqref{pde Brownian meander 2}.
\end{proof}
Relationship \eqref{Telegraph meander t 4} and the fact that $p_{\lambda,c}^{-}$ satisfies the telegraph equation suggest to apply the same technique exploited in the proof of previous theorem in order to derive the equation that governs the law of the telegraph meander. Indeed, in the next theorem it is proved that the law of the telegraph meander satisfies an hyperbolic equation in which is contained the telegraph operator $\frac{\partial^2}{\partial t^2}-c^2\frac{\partial^2}{\partial x^2}+2\lambda\frac{\partial}{\partial t}$.
\begin{theorem}
Let $T^+ = (T^+(s))_{s \in [0,t]}$ be a telegraph meander. Then the absolutely continuous component of the law of $T^+(t)$ satisfies 
\begin{equation}\label{pde telegraph meander}
\frac{\partial^2u}{\partial t^2}-c^2\frac{\partial^2 u}{\partial x^2}+2\lambda\frac{\partial u}{\partial  t}-\frac{2a}{t}\frac{\partial u}{\partial t}=\Big(\frac{\lambda}{t}-\frac{2a}{ t^2}\Big)u
\end{equation}
where 
\begin{equation}
a(t) = \frac{I_1(\lambda t)}{I_0(\lambda t)+I_1(\lambda t)}.
\end{equation}
\end{theorem}
\begin{proof}
We rewrite equation \eqref{Telegraph meander t 4} in a suitable way
\begin{equation}\label{Telegraph meander t 5}
q_{\lambda,c}(x,t)p^{-}_{\lambda,c}(0,t) = -\frac{\partial}{\partial x}p^{-}_{\lambda,c}(x,t).
\end{equation}
Hence, by applying the second time derivative on both member of \eqref{Telegraph meander t 5} we obtain
\begin{align}\label{cp pde Telegraph meander 1}
&\frac{\partial^2}{\partial t^2}q_{\lambda,c}(x,t)p^{-}_{\lambda,c}(0,t)+2\frac{\partial}{\partial t}q_{\lambda,c}(x,t)\frac{\partial}{\partial t}p^{-}_{\lambda,c}(0,t)+q_{\lambda,c}(x,t)\frac{\partial^2}{\partial t^2}p^{-}_{\lambda,c}(0,t) =-\frac{\partial^3}{\partial t^2\partial x}p^{-}_{\lambda,c}(x,t)
\end{align}
Analogously, by taking in both members of \eqref{Telegraph meander t 5} the second order space derivative and also multiplying with $-c^2$, we get
\begin{align}\label{cp pde Telegraph meander 2}
-c^2\frac{\partial^2}{\partial x^2}q_{\lambda,c}(x,t)p^{-}_{\lambda,c}(0,t) = c^2\frac{\partial^3}{\partial x^3}p^{-}_{\lambda,c}(x,t).
\end{align}
We sum up equations \eqref{cp pde Telegraph meander 1} and \eqref{cp pde Telegraph meander 2} and it follows that
\begin{align}\label{cp pde Telegraph meander 3}
&\frac{\partial^2}{\partial t^2}q_{\lambda,c}(x,t)p^{-}_{\lambda,c}(0,t)+2\frac{\partial}{\partial t}q_{\lambda,c}(x,t)\frac{\partial}{\partial t}p^{-}_{\lambda,c}(0,t)+q_{\lambda,c}(x,t)\frac{\partial^2}{\partial t^2}p^{-}_{\lambda,c}(0,t) -c^2\frac{\partial^2}{\partial x^2}q_{\lambda,c}(x,t)p^{-}_{\lambda,c}(0,t)\nonumber\\&=-\frac{\partial}{\partial x}\Big\{\frac{\partial^2}{\partial t^2}p^{-}_{\lambda,c}(x,t)-\frac{\partial^2}{\partial x^2}c^2p^{-}_{\lambda,c}(x,t)\Big\}.
\end{align}
Now, the density function $p^{-}_{\lambda,c}(x,t)$ satisfies the telegraph equation
\begin{equation}\label{Telegraph equation}
\frac{\partial^2}{\partial t^2}p^{-}_{\lambda,c}(x,t) +2\lambda\frac{\partial}{\partial t}p^{-}_{\lambda,c}(x,t) =c^2\frac{\partial^2}{\partial x^2}p^{-}_{\lambda,c}(x,t),
\end{equation}
therefore, by means of \eqref{Telegraph meander t 5} and \eqref{Telegraph equation}, the right hand side of equation \eqref{cp pde Telegraph meander 3} becomes 
\begin{align*}
&2\lambda\frac{\partial^2}{\partial x\partial t}p^{-}_{\lambda,c}(x,t) = -2\lambda \frac{\partial}{\partial t}\Big\{q_{\lambda,c}(x,t)p^{-}_{\lambda,c}(0,t)\Big\}=-2\lambda\big\{\frac{\partial}{\partial t}q_{\lambda,c}(x,t)p^{-}_{\lambda,c}(0,t)+q_{\lambda,c}(x,t)\frac{\partial}{\partial t}p^{-}_{\lambda,c}(0,t)\big\}.
\end{align*}
As a result, equation \eqref{cp pde Telegraph meander 3} takes the form 
\begin{align}\label{cp pde Telegraph meander 5}
&\Big\{\frac{\partial^2}{\partial t^2}q_{\lambda,c}(x,t)-c^2\frac{\partial^2}{\partial x^2}q_{\lambda,c}(x,t)+2\lambda\frac{\partial}{\partial t}q_{\lambda,c}(x,t)\Big\}p^{-}_{\lambda,c}(0,t)+2\frac{\partial}{\partial t}q_{\lambda,c}(x,t)\frac{\partial}{\partial t}p^{-}_{\lambda,c}(0,t)\nonumber \\
&=-q_{\lambda,c}(x,t)\big\{\frac{\partial^2}{\partial t^2}p^{-}_{\lambda,c}(0,t)+2\lambda\frac{\partial}{\partial t}p^{-}_{\lambda,c}(0,t)\big\}.
\end{align}
We now compute the first and second time derivative of $p^{-}_{\lambda,c}(0,t)$. 
\begin{align}
&\frac{\partial}{\partial t}p^{-}_{\lambda,c}(0,t) = -\frac{\lambda^2}{2c}e^{-\lambda t}(I_0(\lambda t)+I_1(\lambda t))+\frac{\lambda^2}{2c}e^{-\lambda t}(I_0(\lambda t)+I_1(\lambda t)-\frac{1}{\lambda t}I_1(\lambda t)) \nonumber\\
&=-\frac{\lambda}{2ct}e^{-\lambda t}I_1(\lambda t) \label{First Derivative p-}\\
&\frac{\partial^2}{\partial t^2}p^{-}_{\lambda,c}(0,t) = \frac{\lambda}{2ct^2}e^{-\lambda t}I_1(\lambda t)+\frac{\lambda^2}{2ct}e^{-\lambda t}I_1(\lambda t)-\frac{\lambda^2}{2ct}e^{-\lambda t}(I_0(\lambda t)-\frac{1}{\lambda t}I_1(\lambda t)) \nonumber\\&=\frac{\lambda}{ct^2}e^{-\lambda t}I_1(\lambda t)+\frac{\lambda^2}{2ct}e^{-\lambda t}(I_1(\lambda t)-I_0(\lambda t)) \label{Second Derivative p-} ,\end{align}
where we made use of \eqref{properties Bessel} and \eqref{properties derivative Bessel}. Finally,
\begin{align}\label{cp Telegraph meander 6}
&\frac{\partial^2}{\partial t^2}p^{-}_{\lambda,c}(0,t)+2\lambda\frac{\partial}{\partial t}p^{-}_{\lambda,c}(0,t) = \frac{\lambda}{2c}e^{-\lambda t}\big\{2t^{-2}I_1(\lambda t)-\lambda t^{-1}\big(I_0(\lambda t)+I_1(\lambda t)\big)\big\}.
\end{align}
Therefore, by dividing with $p^{-}_{\lambda,c}(0,t)$ and substituting \eqref{First Derivative p-}, \eqref{Second Derivative p-} and \eqref{cp Telegraph meander 6} into \eqref{cp pde Telegraph meander 5}, we arrive at
\begin{align*}\label{pde Telegraph meander 2}
&\frac{\partial^2}{\partial t^2}q_{\lambda,c}(x,t) - c^2\frac{\partial^2}{\partial x^2}q_{\lambda,c}(x,t) +2\lambda\frac{\partial}{\partial t}q_{\lambda,c}(x,t)-\frac{2I_1(\lambda t)}{t(I_0(\lambda t)+I_1(\lambda t))} \frac{\partial}{\partial t}q_{\lambda,c}(x,t) \nonumber\\
&=-q_{\lambda,c}(x,t)\Big(\frac{2I_1(\lambda t)}{t^2(I_0(\lambda t)+I_1(\lambda t))}-\frac{\lambda}{t}\Big),
\end{align*}
that is \eqref{pde telegraph meander}.
\end{proof}
Under the so called Kac's scaling condition, namely when $\lambda \to +\infty$ and $c\to +\infty$ such that $c^2/\lambda \to 1$, the telegrapher's equation becomes the heat equation. As one would expect, under the same condition, the equation governing the telegraph meander turns into the equation governing Brownian meander. Indeed, we rewrite equation \eqref{pde telegraph meander} as
\begin{equation*}\label{pde telegraph meander 3}
-\frac{1}{2\lambda}\frac{\partial^2u}{\partial t^2}+\frac{c^2}{2\lambda}\frac{\partial^2 u}{\partial x^2}-\frac{\partial u}{\partial  t}+\frac{a_{\lambda}}{\lambda t}\frac{\partial u}{\partial t}=-\Big(\frac{1}{2t}-\frac{a_{\lambda}}{ \lambda t^2}\Big)u
\end{equation*}
where we added the dependence of $a$ with respect to $\lambda$. Then, from the asymptotic expansion of Bessel functions
\begin{equation}\label{Asymptotic bessel}
I_{\nu}(x) \sim \frac{e^x}{\sqrt{2\pi x}}, \;\; x \to +\infty, \;\; \nu \in \mathbb{R}
\end{equation}
we see that as $\lambda \to +\infty$, $a_\lambda \to 1/2$ and the previous equation becomes the equation governing the law of Brownian meander.
\section{Weak convergence of telegraph meander}
Here we study the convergence of the distribution of the telegraph meander under the Kac's condition by setting $\alpha= \lambda =c^2$. First, we consider the convergence of the finite dimensional distributions. By means of the asymptotic estimate of the Bessel functions, it can be proved that the density function of the telegraph process converges to the density of Brownian motion. The same technique is applied in the next theorem.
\begin{theorem}
Let $T^+ = (T^+(s))_{s \in [0,t]}$ be a telegraph meander. Then, for every $n\ \in\mathbb{N}$, $0<t_1<...<t_n\leq n$, $x_1,...,x_n \in \mathbb{R}$
\begin{equation}
\lim_{\alpha \to +\infty}\mathsf{P}\{T^+_{\alpha}(t_1) \in dx_1,...,T^+_{\alpha}(t_n) \in dx_n\} = \mathsf{P}\{B^+(t_1) \in dx_1,...,B^{+}(t_n) \in dx_n\},
\end{equation}
where $B^+$ is Brownian meander on $[0,t]$.
\end{theorem}
\begin{proof}
By resorting to the asymptotic estimate of the Bessel functions (formula \eqref{Asymptotic bessel}) and applying it on \eqref{minimum positive}, \eqref{minimum negative} and the law expressed in Lemma \ref{lemma joint distribution telegraph minimum with velocity}, after easy computations is it possible to see that, for $v \in \{-\sqrt{\alpha},\sqrt{\alpha}\}$
\begin{align}
&\lim_{\alpha \to +\infty}\mathsf{P}\{\min_{0\leq s \leq t}T_\alpha(s) \geq -\beta|V(0) = v\} = \int_0^\beta\frac{2e^{-\frac{x^2}{2t}}}{\sqrt{2\pi t}}dx = \mathsf{P}\{\min_{0\leq s \leq t}B(s) \geq -\beta\},\\
&\lim_{\alpha \to +\infty}\mathsf{P}\{T_\alpha(t) \in dx,V(t) = v,\min_{0\leq s \leq t}T_\alpha(s) \geq -\beta|V(0) = v\} \nonumber\\&=\frac{1}{2}\bigg\{\frac{e^{-\frac{x^2}{2t}}}{\sqrt{2\pi t}}-\frac{e^{-\frac{(2\beta+x)^2}{2t}}}{\sqrt{2\pi t}}\bigg\}\mathsf{1}_{[-x\lor 0,+\infty)}(\beta)dx =\frac{1}{2}\mathsf{P}\{B(t) \in dx,\min_{0\leq s\leq t}B(s)\geq -\beta\},\\
&\lim_{\alpha \to +\infty}\frac{\mathsf{P}\{T_\alpha(s) \in dx,V(s) = v,\min_{0\leq u \leq s}T_\alpha(u) \geq 0|V(0) = \sqrt{\alpha}\}}{\mathsf{P}\{\min_{0\leq s \leq t}T_\alpha(s) \geq 0|V(0) = \sqrt{\alpha}\}} = \frac{1}{2}\frac{x}{s}e^{-\frac{x^2}{2s}}\sqrt{\frac{t}{s}}\mathsf{1}_{[0,+\infty)}(x)dx.
\end{align}
Now, by taking the limit of \eqref{fdd} and applying the previous results we get, for $x_1,...,x_n>0$,
\begin{align*}
&\lim_{\alpha \to +\infty}\mathsf{P}\{T^{+}_\alpha(t_1) \in dx_1,...,T^{+}_\alpha(t_n) \in dx_n\} \\
&= \frac{x_1}{t_1}e^{-\frac{x_1^2}{2t_1}}\sqrt{\frac{t}{s_1}}dx_1\prod_{k=2}^n\big(\frac{e^{-\frac{(x_{k}-x_{k-1})^2}{2(t_k-t_{k-1})}}-e^{-\frac{(x_{k}+x_{k-1})^2}{2(t_k-t_{k-1})}}}{\sqrt{2\pi(t_k-t_{k-1})}}\big)dx_k\int_{0}^{x_n}\frac{2e^{-\frac{y^2}{2(t-t_n)}}}{\sqrt{2\pi(t-t_n)}}dy,
\end{align*}
which is the joint distribution of Brownian meander. 
\end{proof}
We now establish a stronger result concerning the weak convergence of the probability measure induced by the telegraph meander to that induced by Brownian meander. Our aim is to provide a conditional version of the weak convergence of the telegraph process to Brownian motion on the set $C^+[0,t]= \{x \in C[0,t]: x\geq 0\}$. Our approach is based on the following equivalent representation of the telegraph process
\begin{equation}\label{Telegraph process definiton 2}
T_\alpha(t) = \frac{1}{\sqrt{\alpha}}\sum_{k = 1}^{\nu(\alpha t)}\xi_0(-1)^{k-1}\xi_k+\frac{\xi_0(-1)^{\nu(\alpha t)}}{\sqrt{\alpha}}(\alpha t-S_{\nu(\alpha t)}), \;\; t \geq 0,
\end{equation}
where $\nu = (\nu(t))_{t \geq 0}$ is a Poisson process of unitary rate, $\xi_0 \in \{-1,1\}$ uniformly and independent of $\nu$, $\{\xi_k\}_{k \geq 1}$ are the i.i.d. exponential inter-arrivals of parameter one associated with $\nu$, and $S_{\nu(\alpha t)} = \sum_{k=1}^{\nu(\alpha t)}\xi_k$. 

Equation \eqref{Telegraph process definiton 2} shows that the telegraph processes can be expressed as a sum of two processes, the first one represents a random walk with a random number of summands, while the second one can be viewed as an interpolating term. Therefore, our strategy is essentially based on proving that a functional conditional central limit theorem holds for the first term and the second term converges to zero as $\alpha \to +\infty$. In such a way, we extend the random function $T_\alpha$ from $C = C[0,1]$ to $D = D[0,1]$, the space of cadlag functions on $[0,1]$ endowed with the Skorohod topology. For a description of this space and its metric topology, one may refer to Chapter 3 of \cite{Billingsley}. We indicate with $\mathcal{C}$ and $\mathcal{D}$, the sigma algebras generated by the uniform and Skorohod topology respectively. The identity map $i:C\to D$ is continuous, hence $\mathcal{D}/\mathcal{C}$ measurable. Then, $i\circ T_{\alpha}$ represents a telegraph process defined on $(D,\mathcal{D})$. For simplicity we denote by $T_\alpha$ the random element on the set of cadlag functions. Let $D^+ = \{x\in D: x\geq 0\}$ be the set of non negative cadlag functions and $\Gamma_{\alpha} = \{T_\alpha \in D^+\}$. Then $T^+_\alpha$ is the restriction of $T_\alpha$ on $\Gamma_\alpha$, namely, the telegraph meander. Hence, $T^+_\alpha: \Gamma_\alpha \to D^+$ is a random function on $(D^+,\mathcal{D^+})$, where $\mathcal{D^+} = D^+ \cap \mathcal{D}$. Let 
\begin{equation}
\hat{T}_{\alpha}(t) = \frac{1}{\sqrt{\alpha}}\sum_{k=1}^{\nu(\alpha t)}\xi_0(-1)^{k-1}\xi_k
\end{equation}
and
\begin{equation}
\Tilde{T}_{n}(t) = \frac{1}{\sqrt{n}}\sum_{k=1}^{\lfloor n t\rfloor}\xi_0(-1)^{k-1}\xi_k \;\; n \in \mathbb{N}
\end{equation}
then, the random variables $\xi_0(-1)^{k-1}\xi_k$ are i.i.d. such that $\mathsf{E}[\xi_0(-1)^{k-1}\xi_k] = 0$, $\mathsf{V}[\xi_0(-1)^{k-1}\xi_k] = 1$. Let also $\tau_0 = \inf\{n \in \mathbb{N}: \Tilde{T}_{n}(1) < 0\}$ and set with \mbox{$\hat{\Gamma}_{\alpha} = \{\tau_0 > \nu(\alpha)\} = \{\hat{T}_\alpha \in D^+\}$}. If $\hat{T}^+_{\alpha}$ is the restriction of $\hat{T}_\alpha$ on $\hat{\Gamma}_{\alpha}$, the hypothesis of Theorem $4.8$ of \cite{I 1974} are clearly satisfied and therefore we can claim that $\hat{T}^+_{\alpha} \Rightarrow B^{+}$ on $D^+$. This is equivalent to state that, for every $A \in \mathcal{D}^+$, such that $\mathsf{P}\{B^+ \in \partial A\} = 0$,
\begin{equation}\label{weak convergence T hat+}
\lim_{\alpha \to +\infty}\mathsf{P}\{\hat{T}_\alpha \in A|\hat{T}_\alpha \in D^+\} = \mathsf{P}\{B^+ \in A\}.
\end{equation}
Our task is to show that from \eqref{weak convergence T hat+} we can derive the weak convergence of $T_{\alpha}^+$ to $B^+$. This can be made with some preliminary results. 
\begin{lemma}\label{Lemma 1 weak convergence}
Let $\delta_0$ be the unit mass at $0$, the identically null function. For every $A \in \mathcal{D^+}$ such that $\delta_0(\partial A) = 0$, it holds 
\begin{align}
&\lim_{\alpha \to +\infty}\mathsf{P}\{T_\alpha-\hat{T}_\alpha \in A|T_\alpha \in D^+\} = \lim_{\alpha \to +\infty}\mathsf{P}\{T_\alpha-\hat{T}_\alpha \in A|\hat{T}_\alpha \in D^+\} = \delta_0(A)
\end{align}
\end{lemma}
\begin{proof}
According to formula \eqref{Minimum non negative} and Lemma $4.2$ of \cite{I 1974},
\begin{align*}
&\mathsf{P}\{\hat{T}_{\alpha} \in D^+\} \sim \frac{C}{\sqrt{\pi\alpha}} \;\;\;\;\;\; \mathsf{P}\{{T}_{\alpha} \in D^+\} \sim \frac{1}{\sqrt{2\pi\alpha}} \;\; \alpha \to +\infty
\end{align*}
where $C$ is a positive constant. Therefore, the thesis will follow if we prove that, 
\begin{equation}\label{T-T hat converges to 0}
\lim_{\alpha \to +\infty}\sqrt{\alpha}\mathsf{P}\{\sup_{t \in [0,1]}|T_\alpha(t)-\hat{T}_\alpha(t)|\geq \epsilon\} = 0 \;\; \epsilon>0.
\end{equation}
Now, by looking at formula \eqref{Telegraph process definiton 2}, the random variable $\alpha t- S_{\nu(\alpha t)}$ defines the so called current life of the renewal process $\nu$, and we can see that is less than $\xi_{\nu(\alpha t)+1}$, the $(\nu(\alpha t)+1)-$th inter-arrival. Thus, we can state the following inequality
\begin{align}\label{current life}
\sup_{t \in [0,1]}|T_\alpha(t)-\hat{T}_\alpha(t)|\leq \frac{1}{\sqrt{\alpha}}\sup_{t \in [0,1]}\xi_{\nu(\alpha t)+1} = \frac{1}{\sqrt{\alpha}}\max_{k=1,...,\nu(\alpha)}\xi_{k+1}.
\end{align}
Hence, by Markov's inequality, for every $\epsilon>0$ and $\eta>0$,
\begin{align*}
&\sqrt{\alpha}\mathsf{P}\{\sup_{t \in [0,1]}|T_\alpha(t)-\hat{T}_\alpha(t)|\geq \epsilon\} \leq \sqrt{\alpha}\mathsf{P}\{\max_{k=1,...,\nu(\alpha)}\xi_{k+1}\geq \sqrt{\alpha}\epsilon\} \\
&\leq \sqrt{\alpha}\mathsf{P}\{\max_{k=1,...,\nu(\alpha)}\xi_{k+1}\geq \sqrt{\alpha}\epsilon, |\nu(\alpha) -\alpha| \leq\alpha\eta\}+\sqrt{\alpha}\mathsf{P}\{|\nu(\alpha) -\alpha| >\alpha\eta\}\\
&\leq \sqrt{\alpha}\mathsf{P}\{\max_{k=1,...,\lfloor\alpha (\eta+1)\rfloor}\xi_{k+1}\geq \sqrt{\alpha}\epsilon\}+\sqrt{\alpha}\mathsf{P}\{|\nu(\alpha) -\alpha| >\alpha\eta\} \leq \frac{\lfloor\alpha (\eta+1)\rfloor}{\alpha\epsilon^3}\mathsf{E}[\xi_{1}^{3}\mathsf{1}_{\{\xi_{1}\geq \sqrt{\alpha}\epsilon\}}] + \frac{1}{\sqrt{\alpha}\eta^2}.
\end{align*}
Because $\xi_{1}$ has finite third moment, \eqref{T-T hat converges to 0} follows from an application of the dominate convergence theorem.
\end{proof}
\begin{lemma}\label{Lemma 2 weak convergence}
The following limit holds
\begin{equation}
\lim_{\alpha \to +\infty}\frac{\mathsf{P}\{T_\alpha \in D^+\}}{\mathsf{P}\{\hat{T}_\alpha \in D^+\}} = 1
\end{equation}
\end{lemma}
\begin{proof}
It is clear that, if $x,y \in D^+$, then $x+y \in D^+$. Thus, we have the following double inequality
\begin{align*}
&\mathsf{P}\{\hat{T}_\alpha \in D^+\} -\mathsf{P}\{\hat{T}_\alpha \in D^+,T_\alpha- \hat{T}_\alpha \notin D^+\} \leq \mathsf{P}\{T_\alpha \in D^+\} 
\leq \mathsf{P}\{\hat{T}_\alpha \in D^+\} + \mathsf{P}\{T_\alpha \in D^+,\hat{T}_\alpha- T_\alpha \notin D^+\},
\end{align*}
from which, it follows that
\begin{align}\label{Lemma 2 inequaility}
&1- \mathsf{P}\{T_\alpha- \hat{T}_\alpha \notin D^+|\hat{T}_\alpha \in D^+\} \leq \frac{\mathsf{P}\{T_\alpha \in D^+\}}{\mathsf{P}\{\hat{T}_\alpha \in D^+\}} \leq \frac{1}{1-\mathsf{P}\{\hat{T}_\alpha- T_\alpha \notin D^+|T_\alpha \in D^+\}}.
\end{align}
By passing to the limit as $\alpha \to +\infty$ on both inequalities of \eqref{Lemma 2 inequaility}, according to Lemma \ref{Lemma 1 weak convergence}, we have
\begin{align*}
1- \delta_0(D\setminus D^+) \leq \lim_{\alpha \to +\infty}\frac{\mathsf{P}\{T_\alpha \in D^+\}}{\mathsf{P}\{\hat{T}_\alpha \in D^+\}} \leq \frac{1}{1-\delta_0(D\setminus D^+)}.
\end{align*}
This is possible because $0 \in \partial D^+$ and hence $\delta_0(\partial(D\setminus D^+)) = 0$. Finally, from $\delta_0(D\setminus D^+) = 0$, the thesis follows.
\end{proof}
Now, we state a slightly modification of Theorem 3.1 of \cite{Billingsley}.
\begin{theorem}\label{Conditional version Billingsley}
Suppose that $(X_n,Y_n)$ are random elements from a probability space $(\Omega,\mathcal{F},\mathsf{P})$ to $S\times S$, where $(S,\rho)$ is an arbitrary metric space. Let $\{\Gamma_n\}$ be a sequence of sets in $\mathcal{F}$ such that $\mathsf{P}\{\Gamma_n\}>0$. Under the condition that
\begin{equation}
\lim_{n \to +\infty}\mathsf{P}\{\rho(X_n,Y_n) \geq \epsilon|\Gamma_n\} = 0
\end{equation}
for every $\epsilon>0$, then
\begin{equation}
\limsup_{n \to +\infty}\mathsf{P}\{X_n \in F |\Gamma_n\} \leq \mathsf{P}\{X \in F\}
\end{equation} 
implies that
\begin{equation}
\limsup_{n\to+\infty}\mathsf{P}\{Y_n \in F|\Gamma_n\} \leq \mathsf{P}\{X \in F\}
\end{equation}
for every closed set $F$ in the topology of $S$.
\end{theorem}
In other words, Theorem \ref{Conditional version Billingsley} says that, conditionally on $\Gamma_n$, if $X_n$ converges weakly and the distance between $X_n$ and $Y_n$ tends to zero, then, conditionally on $\Gamma_n$, $Y_n$ has the same weak limit of $X_n$. 

In our context, Theorem \ref{Conditional version Billingsley}, Lemma \ref{Lemma 1 weak convergence} and \eqref{weak convergence T hat+} implies, conditionally on $\hat{\Gamma}_\alpha$, that $\hat{T}_\alpha$ and $T_\alpha$ have the same weak limit. This means that
\begin{equation}\label{weak convergence T+}
\lim_{\alpha \to +\infty}\mathsf{P}\{T_\alpha \in A|\hat{T}_\alpha \in D^+\} = \mathsf{P}\{B^+ \in A\} 
\end{equation}
for every $A \in \mathcal{D}^+$ such that $\mathsf{P}\{B^+ \in \partial A\} = 0$. 
We are now able to prove the main result of this section.
\begin{theorem}
Let $T_\alpha^+$ be a telegraph meander and $B^+$ be Brownian meander. Then, $T_\alpha^+ \Rightarrow B^+$ on $C^+[0,t]$.
\end{theorem}
\begin{proof}
For any set $A \in \mathcal{D}^+$ the following double inequality is true
\begin{align*}
&\mathsf{P}\{T_\alpha \in A, \hat{T}_\alpha \in D^+\} -\mathsf{P}\{\hat{T}_\alpha \in D^+,T_\alpha- \hat{T}_\alpha \notin D^+\} \\
&\leq \mathsf{P}\{T_\alpha \in A, T_\alpha \in D^+\} \leq \mathsf{P}\{T_\alpha \in A, \hat{T}_\alpha \in D^+\}+\mathsf{P}\{T_\alpha \in D^+, \hat{T}_\alpha- T_\alpha \notin D^+\}.
\end{align*}
Thus
\begin{align*}
&(\mathsf{P}\{T_\alpha \in A| \hat{T}_\alpha \in D^+\} -\mathsf{P}\{T_\alpha- \hat{T}_\alpha \notin D^+|\hat{T}_\alpha \in D^+\})\frac{\mathsf{P}\{\hat{T}_\alpha \in D^+\}}{\mathsf{P}\{T_\alpha \in D^+\}} \\
&\leq \mathsf{P}\{T_\alpha \in A| T_\alpha \in D^+\} \leq \mathsf{P}\{T_\alpha \in A|\hat{T}_\alpha \in D^+\}\frac{\mathsf{P}\{\hat{T}_\alpha \in D^+\}}{\mathsf{P}\{T_\alpha \in D^+\}}+\mathsf{P}\{\hat{T}_\alpha- T_\alpha \notin D^+|T_\alpha \in D^+\}
\end{align*}
By taking the limit as $\alpha \to +\infty$, as a consequence of \eqref{weak convergence T+}, Lemma \ref{Lemma 1 weak convergence} and Lemma \ref{Lemma 2 weak convergence},
\begin{align*}
&\mathsf{P}\{B^+ \in A\} -\delta_0(D\setminus D^+) \leq \lim_{\alpha \to +\infty}\mathsf{P}\{T_\alpha \in A| T_\alpha \in D^+\} \leq \mathsf{P}\{B^+ \in A \}+\delta_0(D\setminus D^+),
\end{align*}
for any $A \in \mathcal{D}^+$ such that $\mathsf{P}\{B^+ \in \partial A\} = 0$. Thus, we obtain that $T^+_\alpha \Rightarrow B^+$ on $D^+$. Finally, by denoting with $C^+ = \{x \in C: x \geq 0\}$, from the fact that $\mathsf{P}\{T^+_\alpha \in C^+\} \equiv \mathsf{P}\{B^+ \in C^+\} = 1$, example 2.9 of \cite{Billingsley}, shows that $T^+_\alpha \Rightarrow B^+$ on $C^+$, from which the thesis can be achieved.
\end{proof}
\section{Conditioned distribution of the telegraph meander}
Here we derive the distribution of telegraph meander at time $t$ conditioned on the number of Poisson events. From \eqref{symmetry property} and the fact that
\begin{align*}
&\mathsf{P}\{\min_{0\leq s\leq t}T(s) = 0|N(t) = n,V(0) = c\} = \mathsf{P}\{\min_{0\leq s\leq t}T(s) \geq 0|N(t) = n,V(0) = c\}, \;\; n\in \mathbb{N}_0
\end{align*}
formulas (4.4) and (4.10) of \cite{CO 2021} can be restated as
\begin{align}\label{Maximum conditioned atom}
&\mathsf{P}\{\min_{0\leq s\leq t}T(s)\geq 0|N(t) = n,V(0) = c\} = 
\begin{cases}
1 &n=0\\
\frac{\binom{2k}{k}}{2^{2k}} &n = 2k,\;\; k \in \mathbb{N}\\
\frac{\binom{2k+1}{k}}{2^{2k+1}} &n = 2k+1,\;\; k \in \mathbb{N}_0.
\end{cases}
\end{align}
Moreover, from equations \eqref{joint pos vel 1} and \eqref{joint pos vel 2}, setting $\beta = 0$, we can write
\begin{align}\label{joint conditioned distribution telegraph with minimum}
&\mathsf{P}\{T(t) \in dx,\min_{0\leq s\leq t}T(s)\geq 0|N(t) = n,V(0) = c\} \nonumber \\
&=\begin{cases}
\mathsf{1}_{\{ct\}}(x)dx &n=0\\
\frac{(2k)!}{k!(k-1)!}\frac{(c^2t^2-x^2)^{k-1}}{(2ct)^{2k}}2x\mathsf{1}_{[0,ct)}(x)dx &n = 2k, \;\;k \in \mathbb{N}\\
\binom{2k+1}{k}\frac{(c^2t^2-x^2)^{k}}{(2ct)^{2k+1}}\frac{ct+(2k+1)x}{ct+x}\mathsf{1}_{[0,ct)}(x)dx &n = 2k+1, \;\;k \in \mathbb{N}_0
\end{cases}
\end{align}
By combining \eqref{Maximum conditioned atom} and \eqref{joint conditioned distribution telegraph with minimum}, we obtain
\begin{align}\label{Telegraph meander conditioned 1}
&\mathsf{P}\{T^+(t) \in dx|N(t) = n\} 
=\begin{cases}
\mathsf{1}_{\{ct\}}(x)dx &n=0\\
2kx\frac{(c^2t^2-x^2)^{k-1}}{(ct)^{2k}}\mathsf{1}_{[0,ct)}(x)dx, &n = 2k,\;\; k \in \mathbb{N}\\
\frac{ct+(2k+1)x}{(ct)^{2k+1}}\frac{(c^2t^2-x^2)^{k}}{ct+x}\mathsf{1}_{[0,ct)}(x)dx &n=2k+1,\;\; k \in \mathbb{N}_0. 
\end{cases}
\end{align}
It is immediate to see that $\mathsf{P}\{T^+(t) \in dx|N(t) = n\} \to 0$ pointwise on $[0,ct)$ as $n \to +\infty$. This can be interpreted by saying that, as the number of changes of directions increases, with high probability, the particle does not move away from the origin. Moreover, as simple calculations show, the point of maximum of \eqref{Telegraph meander conditioned 1}, is equal to
\begin{equation}\label{Mode}
\mathsf{M}[T^+(t)|N(t) = n] = 
\begin{cases}
ct &n=0\\
\frac{ct}{\sqrt{2k-1}}, &n = 2k,\;\; k \in \mathbb{N}\\
ct\frac{\sqrt{2k+2}-1}{2k+1} &n=2k+1,\;\; k \in \mathbb{N}_0. 
\end{cases}
\end{equation}
It is worthwhile to observe that
\begin{align}
&\mathsf{P}\{T^+(t) \in dx|N(t) = 2k\}= -\frac{\partial}{\partial x}\frac{(c^2t^2-x^2)^{k}}{(ct)^{2k}} \;\;k \in \mathbb{N}\label{Telegraph meander conditioned even}\\
&\mathsf{P}\{T^+(t) \in dx|N(t) = 2k+1\}=-\frac{\partial}{\partial x}\frac{(c^2t^2-x^2)^{k}(ct-x)}{(ct)^{2k+1}} \;\;k \in \mathbb{N}_0\label{Telegraph meander conditioned odd}
\end{align}
hence, we can easily check that the laws in \eqref{Telegraph meander conditioned 1} are proper density functions. Furthermore, it is possible to derive the distribution function of the conditioned telegraph meander at the end time $t$. For $k \in \mathbb{N}$
\begin{align}
\mathsf{P}\{T^+(t) \leq x|N(t) = 2k\} =\begin{cases}
0 &x< 0\\
1-\Big(1-\frac{x^2}{c^2t^2}\Big)^{k} &0\leq x<ct \\ 
1 &x \geq ct
\end{cases} 
\end{align}
while, for $k \in \mathbb{N}_0$
\begin{align}
\mathsf{P}\{T^+(t) \leq x|N(t) = 2k+1\} =\begin{cases}
0 &x< 0\\
1-\Big(1-\frac{x^2}{c^2t^2}\Big)^{k}\Big(1-\frac{x}{ct}\Big) &0\leq x<ct \\ 
1 &x \geq ct
\end{cases}
\end{align}
From these formulas, we see that the following property holds
\begin{equation}\label{stochastic dominance}
\mathsf{P}\{T^+(t) \geq x|N(t) = 2k+1\} \leq \mathsf{P}\{T^+(t) \geq x|N(t) = 2k\}.
\end{equation}
The previous inequality has the following meaning. The condition that the motion never crosses the zero level requires that the velocity at the initial time must be positive. This implies that, if the motion changes direction an even number of times during the time interval $[0,t]$, the velocity at the end time is positive, while it is negative in the case of and odd number of changes of directions. Thus, the stochastic dominance arises from the fact that the particle is more likely to cross the positive level $x$ when the velocity changes from negative to positive.

We now compute the conditional moments of the telegraph meander.
\begin{proposition}\label{proposition Telegraph meander conditioned moment}
Let $T^+ = (T^+(s))_{s \in [0,t]}$ be a telegraph meander, then for any $p>0$, $n \in \mathbb{N}_0$
\begin{align}\label{Telegraph meander conditioned moments}
&\mathsf{E}[T^+(t)^{p}|N(t) = n] =
\begin{cases}
(ct)^{p} &n=0\\
(ct)^{p}k!\frac{\Gamma(\frac{p}{2}+1)}{\Gamma(\frac{p}{2}+1+k)} &n = 2k,\;  k \in \mathbb{N} \\
(ct)^{p}k!\Big(\frac{\Gamma(\frac{p}{2}+1)}{\Gamma(\frac{p}{2}+1+k)}-\frac{\frac{p}{2}\Gamma(\frac{p+1}{2})}{\Gamma(\frac{p+1}{2}+1+k)}\Big) &n=2k+1,\; k \in \mathbb{N}_0
\end{cases}
\end{align}
\end{proposition}
\begin{proof}
For the even case, with formula \eqref{Telegraph meander conditioned even} at hand, integration by parts gives
\begin{align*} 
&\mathsf{E}[T^+(t)^{p}|N(t) = 2k] = -\int_0^{ct}x^p\frac{\partial}{\partial x}\frac{(c^2t^2-x^2)^{k}}{(ct)^{2k}}dx = \int_0^{ct}px^{p-1}\frac{(c^2t^2-x^2)^{k}}{(ct)^{2k}}dx \\&=\frac{p}{2}(ct)^{p}\int_0^{1}x^{\frac{p}{2}-1}(1-x)^{k}dx = (ct)^{p}k!\frac{\Gamma\Big(\frac{p}{2}+1\Big)}{\Gamma(\frac{p}{2}+1+k)}.
\end{align*}
In the same way, for the odd case, by means of \eqref{Telegraph meander conditioned odd}, we have
\begin{align*}
&\mathsf{E}[T^+(t)^{p}|N(t) = 2k+1] = -\int_0^{ct}x^p\frac{\partial}{\partial x}\frac{(c^2t^2-x^2)^{k}(ct-x)}{(ct)^{2k+1}}dx = \int_0^{ct}px^{p-1}\frac{(c^2t^2-x^2)^{k}(ct-x)}{(ct)^{2k+1}}dx \\&=\frac{p}{2}(ct)^{p}\bigg(\int_0^{1}x^{\frac{p}{2}-1}(1-x)^{k}dx-\int_0^{1}x^{\frac{p-1}{2}}(1-x)^{k}dx\bigg)=(ct)^{p}k!\bigg(\frac{\Gamma(\frac{p}{2}+1)}{\Gamma(\frac{p}{2}+1+k)}-\frac{\frac{p}{2}\Gamma(\frac{p+1}{2})}{\Gamma(\frac{p+1}{2}+1+k)}\bigg).
\end{align*}
\end{proof}
We remark that the difference between the two moments for the even and odd cases tends to be negligible when $n$ is large enough. This indicates that, when the motion changes directions many times, the position is not significantly affected by whether the number of changes of direction is even or odd. Moreover, by applying Stirling's formula to \eqref{Telegraph meander conditioned moments}, it turns out that $n^{p/2}\mathsf{E}[T(t)^{p}|N(t) = n] \to 0$. Therefore, the telegraph meander conditioned on the event $N(t)=n$, converges in $L^p$ to $0$ for any $p>0$ as $n \to +\infty$. And hence this tells us that, when the number of Poisson events is large, on average, the motion concentrates around the origin.

\end{document}